\documentclass[12pt,a4paper,openbib]{article}
\usepackage[utf8]{inputenc}
\usepackage[T1]{fontenc}
\usepackage[french]{babel}
\usepackage{amsmath}
\usepackage{amsfonts}
\usepackage{dsfont}
\usepackage{tikz}
\usetikzlibrary{cd}
\usepackage{amssymb}
\usepackage{amsthm}
\usepackage{lmodern}
\usepackage{fancyhdr}
\usepackage{lettrine}
\usepackage{bm}
\usepackage{enumerate}
\usepackage{multicol}
\usepackage[hidelinks]{hyperref}
\usepackage[nottoc, notlof, notlot]{tocbibind}
\usepackage[left=2cm,right=2cm,top=2cm,bottom=2cm]{geometry} 
\usepackage{graphicx}
\usepackage{mathtools}
\usepackage{wasysym}
\usepackage{stmaryrd}
\usepackage[all,cmtip]{xy}
\usepackage{enumitem}
\usepackage{relsize} 
\usepackage{mathrsfs}

\makeatletter

\newcommand{\address}[1]{\gdef\@address{#1}}
\newcommand{\email}[1]{\gdef\@email{\url{#1}}}
\newcommand{\@endstuff}{\par\vspace{\baselineskip}\noindent\small
\begin{tabular}{@{}l}\scshape\@address\\\textit{E-mail address:} \@email\end{tabular}}
\AtEndDocument{\@endstuff}
\makeatother

\author{Damien Junger\footnote{This work has been written in a great part during the author PhD thesis at the ENS Lyon. His work are currently funded by the Deutsche Forschungsgemeinschaft (DFG, German Research Foundation) under Germany's Excellence Strategy EXC 2044–390685587, Mathematics Münster: Dynamics–Geometry–Structure.}}
\address{Mathematisches Institut, Universität Münster,\\ Fachbereich Mathematik und Informatik der Universität Münster,  Orléans-Ring 10, 48149 Münster, Germany.}
\email{djunger@uni-muenster.de}
\title{Cohomologie de de Rham du revêtement modéré de l'espace de Drinfeld}

\newtheorem{theointro}{Th\'eor\`eme}

\newtheorem{theo}{Th\'eor\`eme}[section]
\newtheorem{lem}[theo]{Lemme}
\newtheorem{coro}[theo]{Corollaire}
\newtheorem{prop}[theo]{Proposition}

\theoremstyle{definition}
\newtheorem{defi}[theo]{D\'efinition}

\theoremstyle{remark}
\newtheorem{rem}[theo]{Remarque}

\DeclareMathOperator{\cind}{c-ind}
\DeclareMathOperator{\ind}{ind}

\DeclareMathOperator{\spf}{Spf}
\DeclareMathOperator{\spg}{Sp}

\DeclareMathOperator{\mat}{M}
\DeclareMathOperator{\gln}{GL}
\DeclareMathOperator{\pgln}{PGL}

\DeclareMathOperator{\dl}{DL}

\DeclareMathOperator{\stab}{Stab}
\DeclareMathOperator{\fix}{Fix}

\DeclareMathOperator{\hhh}{H}
\DeclareMathOperator{\rrr}{R}
\DeclareMathOperator{\homm}{Hom}

\DeclareMathOperator{\ost}{Ost}
\DeclareMathOperator{\fst}{Fst}

\DeclareMathOperator{\id}{Id}

\DeclareMathOperator{\nr}{Nr}

\DeclareMathOperator{\lie}{Lie}
\DeclareMathOperator{\jl}{JL}

 \newcommand{\iso}{\stackrel{\sim}{\fl}}

\font\tengoth=eufb10
\font\sevengoth=eufb7
\font\fivegoth=eufb5

\newfam\gothfam
\textfont\gothfam=\tengoth
\scriptfont\gothfam=\sevengoth
\scriptscriptfont\gothfam=\fivegoth

\def\A{{\mathbb{A}}}
\def\B{{\mathbb{B}}}
\def\C{{\mathbb{C}}}

\def\F{{\mathbb{F}}}

\def\H{{\mathbb{H}}}

\def\N{{\mathbb{N}}}

\def\P{{\mathbb{P}}}
\def\Q{{\mathbb{Q}}}

\def\Z{{\mathbb{Z}}}

\def\BC{{\mathcal{B}}}

\def\HC{{\mathcal{H}}}
\def\IC{{\mathcal{I}}}
\def\GC{{\mathcal{G}}}

\def\MC{{\mathcal{M}}}
\def\OC{{\mathcal{O}}}

\def\TC{{\mathcal{T}}}

\def\XC{{\mathcal{X}}}

\def\SG{{\mathfrak{S}}}

\def\XG{{\mathfrak{X}}}

\def\mG{{\mathfrak{m}}}

\def\Lf{{\mathscr{L}}}

\def\Of{{\mathscr{O}}}

\def\bar#1{\overline{#1}}
\def\het#1{{\rm H}^{#1}_{\rm ét}}
\def\hetc#1{{\rm H}^{#1}_{{\rm ét},c}}

\def\hdr#1{{\rm H}^{#1}_{\rm dR}}
\def\hdrc#1{{\rm H}^{#1}_{{\rm dR},c}}
\def\hrig#1{{\rm H}^{#1}_{\rm rig}}
\def\hrigc#1{{\rm H}^{#1}_{{\rm rig},c}}

\def\et{\text{ et }}

\def\and{\text{ and }}

\def\fl{\rightarrow}

\def\fln#1#2{\xrightarrow[#2]{#1}}

\def\limp{\varprojlim}

\def\som#1#2#3{\sum\limits_{{\substack{#2}}}^{#3}{#1}}
\def\pro#1#2#3{\prod\limits_{{\substack{#2}}}^{#3}{#1}}

\def\inter#1#2#3{\bigcap\limits_{{\substack{#2}}}^{#3}{#1}}
\def\uni#1#2#3{\bigcup\limits_{{\substack{#2}}}^{#3}{#1}}

\def\drt#1#2#3{\bigoplus\limits_{{\substack{#2}}}^{#3}{#1}}

\begin{document}
\maketitle

\begin{abstract}
In this article, we study the De Rham cohomology of the first cover in the Drinfel'd tower. In particular, we get a purely local proof that the supercuspidal part realizes the local Jacquet-Langlands correspondence for ${\rm GL}_n$ by comparing it to the rigid cohomology of some Deligne-Lusztig varieties. The representations obtained are analogous to the ones appearing in the $\ell$-adic cohomology if we forget the action of the Weyl group. The proof relies on the generalization of an excision result of Grosse-Kl\"onne and on the explicit description of the first cover as a cyclic cover obtained by the author on a previous work.
\end{abstract}

\tableofcontents

\section{Introduction} 


  Soit $d\geq 1$ un entier et $p$ un nombre premier, nous étudierons dans cet article la cohomologie de De Rahm de $\H_K^d$ l'espace de Drinfeld de dimension 
  $d$ sur $K$ une extension finie de $\Q_p$. C'est un ouvert de l'espace projective rigide, sur lequel agit naturellement $G=\gln_{d+1}(K)$,  tel que 
  $$\H_K^d(C)=\P_K^d(C) \backslash \bigcup_{H \in \HC} H,$$ o\`u $\HC$ est l'ensemble des hyperplans $K$-rationnels de $\P_K^d$ et $C=\widehat{\bar{K}}$ le complété d'une clôture alg\'ebrique de $K$.
  
  Dans un article monumental \cite{dr2}, Drinfeld a construit 
  une tour de revêtements finis étales $G$-équivariants 
$(\MC_{Dr}^n)_{n\geq 0}$ de l'espace $\MC_{Dr}^0:= \H_{\breve{K}}^d \times \Z$ (avec $\breve{K}= \widehat{K^{nr}}$ le compl\'eté de l'extension maximale non ramifi\'ee dans $\bar{K}$ et $\H_{\breve{K}}^d=\H_K^d\otimes_K \breve{K}$), de groupe de Galois 
$\OC_D^*$ avec $D$ l'algèbre à division sur 
  $K$ de dimension $(d+1)^2$ et d'invariant $1/(d+1)$. 
La cohomologie \'etale $l$-adique pour $l \neq p$ de cette tour fournit une r\'ealisation g\'eom\'etrique de la correspondance de Langlands et de Jacquet-Langlands locales, cf. \cite{Harr}, \cite{Boy}, \cite{harrtay}. Pour $l=p$, 
$d=1$ et $K=\Q_p$, il est montré dans \cite{GPW1} que la cohomologie 
étale $p$-adique de ces espaces encode aussi la correspondance de Langlands 
locale $p$-adique pour ${\rm GL}_2(\Q_p)$.

Dans ce travail, nous nous int\'eressons \`a la cohomologie de De Rham de la tour de 
Drinfeld, guidés par le principe informel suivant: les résultats $l$-adiques 
établis dans \cite{Harr}, \cite{Boy}, \cite{harrtay} doivent avoir des analogues en cohomologie de de Rham, obtenus en oubliant simplement l'action du groupe de Weil 
$W_K$ de $K$ et en changeant les coefficients. L'action de $W_K$ sur les groupes de cohomologie $l$-adique est très intéressante, mais elle devient invisible sur les groupes de cohomologie de de Rham, qui encodent uniquement les liens entre les représentations de $G$ et de $D^*$. Ce genre de résultat a été démontré pour $
d=1$ dans \cite{brasdospi} (pour $K=\Q_p$) et dans \cite{GPW1} (pour $K$ quelconque), pour toute la tour de Drinfeld. Notre résultat principal 
est une preuve (purement locale) de ce principe quand $n=1$ et $d$ est quelconque, plus précisément pour la "partie supercuspidale" de la cohomologie. 
Il s'agit d'un analogue en cohomologie de de Rham du résultat $l$-adique démontré par voie locale par Wang \cite{wa}. Nous utilisons de mani\`ere cruciale les résultats g\'eom\'etriques concernant $\MC_{Dr}^1$ obtenus dans loc.cit (la situation est nettement plus compliquée pour $\MC_{Dr}^n$ quand $n>1$, et il est peu probable qu'une approche purement locale puisse résoudre ce problème).

   Pour énoncer notre résultat principal, nous avons besoin de quelques préliminaires. 
   Le groupe de Galois du revêtement $\MC_{Dr}^1\to \MC_{Dr}^0$
 est $\F_{q^{d+1}}^*$ (avec $\F_q$ le corps résiduel de $K$), un groupe cyclique d'ordre premier à $p$ (ce qui jouera un rôle fondamental par la suite). Soit $\theta: \F_{q^{d+1}}^*\to C^*$ un caractère 
 \emph{primitif} de ce groupe, i.e. qui ne se factorise pas par la norme $\F_{q^{d+1}}^*\to \F_{q^{e}}^*$ pour tout diviseur $e$ de $d+1$, différent de $d+1$. On peut associer à $\theta$ les objets suivants:
 
 $\bullet$ une représentation de Deligne-Lusztig (ou de Green) $\bar{\pi}_{\theta}$ du groupe ${\rm GL}_{d+1}(\F_q)$. 
 
 $\bullet$ une représentation de $D^*$
 $$\rho(\theta):=\ind_{\OC_D^* \varpi^{\Z}}^{D^*} \theta.$$
 
 $\bullet$ une représentation de $G$ 
 $$\jl(\rho(\theta)):=\cind_{\gln_{d+1}(\OC_K)\varpi^{\Z}}^G \bar{\pi}_{\theta}.$$
 La notation est bien entendu inspirée par la correspondance de Jacquet-Langlands pour les représentations supercuspidales de niveau $0$ et de caract\`ere central trivial sur $\varpi^{\Z}$.

\begin{theointro}
\label{theointroprincbis}
  Pour tout caractère primitif $\theta: \F_{q^{d+1}}^*\to C^*$ il existe des isomorphismes de $G$-représentations 
 
$$\homm_{D^*}(\rho(\theta), \hdrc{i}(\MC_{Dr, C}^1/ \varpi^{\Z})){\cong} \begin{cases} \jl(\rho(\theta))^{d+1} &\text{ si } i=d \\ 0 &\text{ sinon} \end{cases}.$$
\end{theointro}   


 
   Expliquons les difficultés qu'il faut surmonter pour prouver ce théorème. La principale   est l'absence d'un modèle semi-stable (ou semi-stable généralisé, cf. ci-dessous) de l'espace $\MC_{Dr}^1$, alors que l'on dispose d'un tel modèle (construit par Deligne) 
   $\widehat{\MC}_{Dr}^0$ pour $\MC_{Dr}^0$ (du côté Lubin-Tate, la situation est meilleure: Yoshida a construit \cite{yosh} un mod\`ele semi-stable g\'en\'eralis\'e du premier revêtement de l'espace de Lubin-Tate et étudié la géométrie de sa fibre spéciale). Le schéma formel $\widehat{\MC}_{Dr}^0$ possède une interprétation modulaire; grâce à un théorème fondamental de Drinfeld \cite{dr2}, ce modèle classifie des déformations par quasi-isogénie d'un 
   $\OC_D$-module formel spécial de hauteur $(d+1)^2$ au sens de Drinfeld. Soit $\XG$ le $\OC_D$-module formel spécial universel sur $\widehat{\MC}_{Dr}^0$, si 
   $\Pi_D$ est une uniformisante de $\OC_D$, alors $\XG[\Pi_D]$ est un sch\'ema formel en $\F_p$-espaces vectoriels de Raynaud dont on connait une classification \cite{Rayn}. De plus, $\widehat{\MC}_{Dr}^1$ est  l'espace obtenu en analytifiant 
   $\XG[\Pi_D]\setminus \{0\}$. 
   
   En utilisant les observations ci-dessus, on peut \'etudier certains ouverts de l'espace $\MC_{Dr}^1$, qui admettent un mod\`ele lisse dont la fibre sp\'eciale est isomorphe \`a une vari\'et\'e de Deligne-Lusztig (cette observation cruciale a été faite en premier par Teitelbaum \cite{teit2} quand $d=1$ et a été généralisée par Wang \cite{wa}). La difficulté est alors de montrer que l'étude de ces ouverts suffit à comprendre la cohomologie de l'espace tout entier. En cohomologie $l$-adique, cela se fait par une étude délicate des cycles \'evanescents pour relier des questions sur $\MC_{Dr}^1$ \`a des questions sur la fibre sp\'eciale et plus particuli\`erement, sur la vari\'et\'e de Deligne-Lusztig. Ces m\'ethodes sont propres \`a la cohomologie $l$-adique et pour surmonter cet obstacle, nous avons besoin du point technique suivant. 
   Soit $\XC$ un schéma formel sur $\OC_K$, localement de type fini. On dit que $\XC$ est de réduction semi-stable généralisée si Zariski-localement sur $\XC$ on peut trouver un morphisme étale vers ${\rm Spf}(\OC_K\langle X_1,...,X_n\rangle/
   (X_1^{\alpha_1}...X_r^{\alpha_r}-\varpi)$ pour certains $r\leq n$ et $\alpha_i\geq 1$ (ou $\varpi$ est une uniformisante de $K$). Si l'on peut choisir les $\alpha_i$ égaux à $1$, on parle de réduction semi-stable. 


\begin{theointro}
\label{theointrolieulissebis}
Soit $\XC$ un schéma formel de réduction semi-stable généralisée, $\XC_s=\uni{Y_i}{i\in I}{}$ la décomposition en composantes irréductibles de la fibre spéciale et $\pi: \TC\rightarrow \XC_\eta$ un rev\^etement \'etale en fibre g\'en\'erique de groupe de Galois $\mu_n$ avec $n\wedge p=1$. 
Si $\TC=\XC_\eta$ ou bien $\XC$ est de réduction semi-stable et $\TC$ est quelconque, 
alors pour toute partie finie $J$ de $I$ on a un isomorphisme naturel \[\hdr{*} (\pi^{-1}(]Y_J[_\XC))\fln{\sim}{} \hdr{*} (\pi^{-1}(]Y_{J}\backslash \bigcup_{i\notin J}Y_i[_\XC))\]
où $ Y_J=\inter{Y_j}{j\in J}{}$.
\end{theointro}

Le cas $n=1$ et $\XC$ semi-stable (non g\'en\'eralis\'e) est d\^u \`a Grosse-Kl\"onne. On a un analogue en cohomologie \'etale $l$-adique en termes de cycles \'evanescents quand $\XC$ est alg\'ebrisable et $| J | =n=1$ d\'emontr\'e par Zheng \cite{zhe} (voir aussi \cite{dat2}). Nous aurons besoin uniquement du cas 
où $\XC$ est semi-stable dans cet article (par contre il est indispensable de travailler avec $J$ et $\TC$ quelconques), mais le cas semi-stable généralisé devrait être utile pour l'étude du premier revêtement de l'espace de Lubin-Tate. 


Au vu de la description de la g\'eom\'etrie de  $\MC_{Dr}^1$ et du th\'eor\`eme \ref{theointrolieulissebis}, nous pouvons \'etablir :

\begin{theointro}\label{theointrohdltdrdlbis}
Soit $\theta$ un caract\`ere primitif de $\F_{q^{d+1}}^*$ et notons $\BC\TC_0$ l'ensemble des sommets de l'immeuble de Bruhat-Tits semi-simple de $G$. On a des isomorphismes 
\[\hdrc{i}(\Sigma^1)[\theta] \cong \bigoplus_{s \in \BC\TC_0} \hrigc{i}(\dl_{\F_q}/ \breve{K})[\theta]\]
o\`u $\dl_{\F_q}$ est la vari\'et\'e de Deligne-Lusztig introduite dans \ref{sssectiondldef}. 
\end{theointro}

Le th\'eor\`eme principal \ref{theointroprincbis} d\'ecoulera alors de la th\'eorie de Deligne-Lusztig que nous rappelons dans la partie \ref{sssectioncohodl}.

Toutefois, pour \'etudier les intersections quelconques de tubes au dessus de composantes $\pi^{-1}(] Y_J [)$, nous aurons besoin de fa\c{c}on essentielle de la description globale du torseur $\Sigma^1 \to \H^d_K$ r\'ealis\'e dans \cite{J2} que l'on restreindra \`a $\pi^{-1}(] Y_J [) \to ] Y_J [$. 


\subsection*{Remerciements}

Le présent travail a été, avec \cite{J1,J2,J3}, en grande partie réalisé durant ma thèse à l'ENS de lyon,  et a pu bénéficier des nombreux conseils et de l'accompagnement constant de mes deux maîtres de thèse Vincent Pilloni et Gabriel Dospinescu. Je les en remercie très chaleureusement.  La preuve du résultat \ref{theodrcouronne} doit beaucoup aux conseils de Laurent Fargues ( pour m'avoir suggéré de voir le complexe de De Rham sur le revêtement comme une somme directe de complexe sur la couronne suivant les parties isotypiques) et de Jean-Claude Sikorav (pour l'explication du cas analogue sur le corps des complexe).

\section{Préliminaires}

\subsection{Conventions générales\label{paragraphconv}}
 On fixe dans tout l'article un nombre premier $p$ et une extension finie $K$ de $\Q_p$, dont on note $\mathcal{O}_K$ l'anneau des entiers, $\varpi$ une uniformisante et $\F=\F_q$ le corps r\'esiduel. On note $C=\hat{\bar{K}}$ le complété d'une clôture algébrique de $K$ et $\breve{K}$ le complété de l'extension maximale non ramifiée de $K$. Soit $L\subset C$ une extension complète de $K$,  susceptible de varier, d'anneau des entiers $\mathcal{O}_L$, d'idéal maximal  $\mG_L$ et de corps r\'esiduel $\kappa$. 

Si $S$ est un $L$-espace analytique, on note $\A^n_{ S}$ (resp. $\P_{ S}^n$) l'espace analytique affine (resp. projectif) sur $S$, de dimension relative $n$. 
  Les espaces $\mathring{\B}^n_S$ et $\B^n_S$ seront les boules unitées ouverte et fermée. 
    



Si $X$ est un $L$-espace analytique et si $F\subset \mathcal{O}(X)$ est une famille finie de fonctions analytiques sur $X$ et $g$ une autre fonction analytique, on note $$X\left(\frac{F}{g}\right)=\left\lbrace x\in X|\,\forall f\in F, |f(x)|\leq |g(x)|\right\rbrace, \,\, X\left(\frac{g}{F}\right)=\left\lbrace x\in X|\,\forall f\in F, |g(x)|\leq |f(x)|\right\rbrace.$$ De m\^eme, $X\left(\left(\frac{F}{g}\right)^{\pm 1}\right)$ ou $X\left(\frac{F}{g},\frac{g}{F}\right)$ d\'esignera $\left\lbrace x\in X|\, \forall f\in F, |f(x)|= |g(x)|\right\rbrace$. Pour $s\in |C^*|$ on note $X\left(\frac{F}{s}\right)=\left\lbrace x\in X|\, \forall f\in F, |f(x)|\leq s\right\rbrace$ (idem pour $X\left(\frac{s}{F}\right)$). 

Les \'el\'ements de la base canonique de $\Z^n$ seront not\'es $(\delta_i)_i$. Si $x=(x_1,\cdots,x_n)$, $\alpha\in \Z^n$ et $A\in \mat_{k,n} (\Z)$, nous adopterons la notation multi-indice, i.e. $x^\alpha=\pro{x_i^{\alpha_i}}{i=1}{n}$ et $x^A=(x^{A\delta_i})_{1\leq i\leq k}$. 


\subsection{Cohomologie étale et torseurs\label{paragraphcohosite}}

Soit $n$ premier à $p$, on identifie $\het{1}(X, \mu_n)$ et l'ensemble des classes d'isomorphisme de $\mu_n$-torseurs sur $X$ (i.e. des revêtements galoisiens $\pi: \TC \to X$ de groupe de Galois $\mu_n$). On note 
$ [\TC]$ la classe d'isomorphisme du torseur $\TC$, vue comme un élément de $\het{1}(X, \mu_n)$. 
Si $\TC_1,\TC_2$ sont des $\mu_n$-torseurs sur $X$, alors $\TC_1 \times_X \TC_2 \to X$ est un rev\^etement de groupe de Galois $\mu_n^2$ et en notant $H\subset \mu_n^2$ l'antidiagonal, le quotient $\TC_3=\TC_1\times_X \TC_2 /H$ est un revêtement de $X$ de groupe de Galois $H\cong\mu_n$ et $[\TC_1]+[\TC_2]= [\TC_3]$. Il est à remarquer que $\TC_1\times_X \TC_2$ est encore un revêtement de $\TC_3$ de groupe de Galois $\mu_n$. 

Le morphisme de Kummer  sera noté $\kappa : \Of^* (X)\to  \het{1}(X,\mu_n)$. 
Le torseur $\kappa(u)$ associé à  une fonction inversible $u$ sur $X$ sera noté $\pi : X(u^{1/N}) \to X$. Si $U\subset X$ est un ouvert affinoïde, alors \[\Of_{X(u^{1/N})} (\pi^{-1}(U))\simeq \Of_X (U) [t]/(t^N -u).\]




\subsection{Cohomologie de De Rham et torseurs \label{paragraphcohoderahm}}
Si $X$ est un affinoïde sur $L$, on notera $X^\dagger$ l’espace surconvergent associé. Si $X^\dagger$ est un espace surconvergent lisse, d'espace rigide analytique sous-jacent $X$, on note $\hdr{*}(X)$ et $\hdr{*}(X^\dagger)$ les hypercohomologies des complexes de de Rham $\Omega_{X/L}^\bullet$ et  $\Omega_{X^\dagger/L}^\bullet$. Par \cite[Proposition 2.5]{GK1}, le théorème $B$ de Kiehl \cite[Satz 2.4.2]{kie} et la suite spectrale de Hodge-de Rham, ces cohomologies sont calculées directement à  partir du complexe de de Rham correspondant, quand $X$ est Stein\footnote{En cohomologie de de Rham (non surconvergente), l'hypothèse $X$ quasi-Stein  suffit.}. Les deux cohomologies coïncident si $X$ est partiellement propre (par exemple Stein). 





 Soit $n$ premier à $p$ et $\pi :\TC \to X$ un $\mu_n$-torseur d'un espace analytique lisse $X$. On a une décomposition $\pi_* \Of_{\TC} = \drt{\Lf_{\chi}}{\chi \in \mu_n^{\lor}}{}$ où $\Lf_\chi$ est l'espace propre associé au caractère $\chi$ pour l'action de $\mu_n$ (c'est un faisceau localement libre de rang 1). Le morphisme $\pi$ étant étale, on a $$\pi_* \Omega_{\TC /L}^q=\Omega_{X/L}^q\otimes \pi_* \Of_{\TC}=\drt{\Omega_{X/L}^q\otimes\Lf_{\chi}}{\chi \in \mu_n^{\lor}}{}$$ et
  de même pour le complexe surconvergent. Les différentielles des complexes de de Rham étant $\mu_n$-équivariantes, elles respectent ces décompositions en somme directe donnant lieu pour chaque $\chi$ à des complexes de cochaînes $\Omega_{X/L}^\bullet\otimes\Lf_{\chi}=:\Omega_{\TC/L}^\bullet[\chi]$ et $\Omega_{X^\dagger/L}^\bullet\otimes\Lf_{\chi}=:\Omega_{\TC^\dagger/L}^\bullet[\chi]$ dont les cohomologies seront notées $\hdr{*}(\TC)[\chi]$ et $\hdr{*}(\TC^\dagger)[\chi]$. C'est la partie isotypique associée à  $\chi$. En particulier, \[\hdr{*}(X)=\hdr{*}(\TC)^{\mu_n}=\hdr{*}(\TC)[1]\et \hdr{*}(X^\dagger)=\hdr{*}(\TC^\dagger)^{\mu_n}=\hdr{*}(\TC^\dagger)[1].\] 

 Soit $\TC_1\to X$, $\TC_2\to Y$ deux $\mu_n$-torseurs sur des $L$-espaces lisses $X$ et $Y$. On obtient $\TC_1\times_L Y$, $X\times_L \TC_2$ deux revêtements sur $X\times_L Y$ et on construit (voir plus haut pour la définition de $H$) \[\TC_3=((\TC_1\times_L Y)\times_{X\times Y} (X\times_L \TC_2)) /H=(\TC_1\times_L \TC_2)/H \] 
 C'est un revêtement de $X\times_L Y$ dont la classe s'identifie à $[\TC_1\times_L Y]+[X\times \TC_2]$. Comme $\TC_1\times_L \TC_2$ est un revêtement de $\TC_3$ étale de groupe de Galois $H$, on a 
\begin{prop}

 On a des isomorphismes naturels 

\[\hdr{q}(\TC_3)=\hdr{q}(\TC_1\times_L \TC_2)^H=\drt{\drt{\hdr{q_1} (\TC_1)[ \chi]\otimes \hdr{q_2} (\TC_2)[ \chi] }{\chi\in \mu_n^\lor }{}}{ q_1 +q_2=q}{} ,\] 
\[ \hdr{q}(\TC_3)[\chi]=\drt{\hdr{q_1} (\TC_1)[ \chi]\otimes \hdr{q_2} (\TC_2)[ \chi] }{ q_1 +q_2=q}{} \]
idem pour la cohomologie de de Rham surconvergente. 
 
 \end{prop} 
 
 \subsection{Cohomologie rigide}

La cohomologie rigide d'un schéma algébrique $Y$ sur le corps résiduel $\kappa$ de $L$ sera notée $\hrig{*} (Y/L)$. On rappelle la dualité de Poincaré:

\begin{prop} \label{theodualitepoinc}
\begin{enumerate}
\item Si $X$ est un $L$-affinoïde lisse, pur de dimension $d$, alors \[\hdr{i} (X^\dagger)\cong \hdrc{2d-i} (X^\dagger)^\lor \et \hdrc{i} (X^\dagger)\cong \hdr{2d-i} (X^\dagger)^\lor.\]

\item   Si $X$ est un $L$-espace lisse et Stein, pur de dimension $d$,  alors \[\hdr{i} (X)\cong \hdrc{2d-i} (X)^\lor \et \hdrc{i} (X)\cong \hdr{2d-i} (X)^\lor.\]
\item Si $Y$ est un schéma lisse sur $\kappa$, pur de dimension $d$, alors  \[\hrig{i} (Y/L)\cong \hrigc{2d-i} (Y/L)^\lor \et \hrigc{i} (Y/L)\cong \hrig{2d-i} (Y/L)^\lor.\]
\end{enumerate}

\end{prop}

\begin{proof} Voir \cite[proposition 4.9]{GK1} pour le premier point, \cite[proposition 4.11]{GK1} pour le second et 
\cite[théorème 2.4]{bert1} pour le dernier. 
\end{proof}

 Le théorème de comparaison suivant nous sera très utile: 
 
\begin{theo}\label{theopurete}
   
Soit $\XC$ un schéma formel lisse $\spf (\OC_L)$, de fibre spéciale $\XC_s$ et de fibre générique $\XC_\eta$,  on a un isomorphisme naturel \[\hdr{*}(\XC_\eta^\dagger)\cong \hrig{*}(\XC_s).\]

\end{theo}

\begin{proof} Il s'agit de
\cite[proposition 3.6]{GK3}.
\end{proof}

\section{Rappels sur la géométrie de l'espace de Drinfeld}

 Nous rappelons quelques résultats standards concernant la géométrie de l'espace symétrique de Drinfeld et nous renvoyons à  (\cite[section 1]{boca}, \cite[sous-sections I.1. et II.6.]{ds}, \cite[sous-section 3.1.]{dat1}, \cite[sous-sections 2.1. et 2.2]{wa}) pour plus de détails.  
On fixe une extension finie $K$ de $\Q_p$, une uniformisante $\varpi$ de $K$ et un entier $d\geq 1$. On note $\F=\F_q$ le corps résiduel de $K$ et $G={\rm GL}_{d+1}(K)$. 

\subsection{L'immeuble de Bruhat-Tits} \label{paragraphbtgeosimpstd}

\label{paragraphbtsimp}
Notons $\BC\TC$ l'immeuble de Bruhat-Tits associé au groupe $\pgln_{d+1}(K)$. Le $0$-squelette $\BC\TC_0$ de l'immeuble est l'ensemble des réseaux de $K^{d+1}$ à  homothétie près, i.e. $\BC\TC_0$ s'identifie à  $G/K^*\gln_{d+1}(\OC_K)$. Un $(k+1)$-uplet de sommets 
$\sigma=\{s_0,\cdots, s_k\}\subset \BC\TC_{0} $ est un $k$-simplexe de $\BC\TC_k$ si et seulement si, quitte à  permuter les sommets $s_i$, on peut trouver pour tout $i$ des réseaux $M_i$   avec $s_i=[M_i]$ tels que \[M_{-1}=\varpi M_k\subsetneq M_0\subsetneq M_1\subsetneq\cdots\subsetneq M_k.\] 



  En posant 
 $$\bar{M}_i=M_i/ M_{-1},$$
 on obtient un drapeau $0\subsetneq\bar{M}_0\subsetneq \bar{M}_1\subsetneq\cdots\subsetneq \bar{M}_k\cong \F^{d+1}$.  On note 
 $d_i={\rm dim}_{\F} (\bar{M}_i)-1$ et $e_i=d_{i+1}-d_{i}.$
  Nous dirons que le simplexe $\sigma$ est de type $(e_0, e_1,\cdots, e_k)$.
  
  Considérons une base $(\bar{f}_0,\cdots, \bar{f}_d)$ adaptée  au drapeau i.e. telle que $\bar{M}_i=\left\langle \bar{f}_0,\cdots , \bar{f}_{d_i}\right\rangle$ pour tout $i$. Pour tout choix de relevés $(f_0,\cdots ,f_d)$ de $(\bar{f}_0,\cdots ,\bar{f}_d)$ dans $M_0$, on a
  $$M_i=\left\langle f_0,\cdots, f_{d_i }, \varpi f_{d_i+1},\cdots, \varpi f_d \right\rangle=N_0\oplus \cdots \oplus N_i\oplus \varpi (N_{i+1} \oplus \cdots\oplus N_k),$$
  où 
  $$N_i=\left\langle f_{d_{i-1}+1},\cdots, f_{d_{i}} \right\rangle$$ avec $d_{-1}=-1$. Si $(f_0,\cdots ,f_d)$ est la base canonique de $K^{d+1}$, nous dirons que $ \sigma $ est le simplexe standard de type $(e_0, e_1,\cdots, e_k)$.


La réalisation topologique de l'immeuble sera notée $|\BC\TC|$. Nous confondrons les simplexes avec leur réalisation topologique de telle manière que $|\BC\TC|=\uni{\sigma}{\sigma\in \BC\TC}{}$. Les différents $k$-simplexes, vus comme des compacts de la réalisation topologique, seront appelés faces. L'intérieur d'une face $\sigma$ sera noté $\mathring{\sigma}=\sigma \backslash \uni{\sigma'}{\sigma'\subsetneq\sigma}{}$ et sera appelé cellule.







\subsection{L'espace des hyperplans $K$-rationnels} 

   On note $\HC$ l'ensemble des hyperplans $K$-rationnels dans $\P^d$. 
Si $a=(a_0,\dots, a_d)\in C^{d+1}\backslash \{0\}$, $l_a$ désignera l'application \[b=(b_0,\dots, b_d)\in C^{d+1} \mapsto \left\langle a,b\right\rangle := \som{a_i b_i}{0\leq i\leq d}{}.\] Ainsi $\HC$ s'identifie à  $\{\ker (l_a),\; a\in K^{d+1}\backslash \{0\} \}$ et à  $\P^d (K)$.
   
   Le vecteur $a=(a_i)_{i}\in C^{d+1}$ est dit unimodulaire si $|a|_{\infty}(:=\max (|a_i|))=1$. L'application 
   $a\mapsto H_a:=\ker (l_a)$ induit une bijection entre le quotient de l'ensemble des vecteurs unimodulaires 
   $a\in K^{d+1}$ par l'action évidente de $\OC_K^*$ et l'ensemble $\mathcal{H}$. 
   
   Pour $a\in K^{d+1}$ unimodulaire et $n\geq 1$, 
   on considère l'application $l_a^{(n)} $ \[b\in (\OC_{C}/\varpi^n)^{d+1}\mapsto \left\langle a,b\right\rangle\in \OC_{C}/\varpi^n\] et on note $$\HC_n =\{\ker (l_a^{(n)}),\; a\in K^{d+1}\backslash \{0\} \; {\rm unimodulaire} \}\simeq\P^d(\OC_K/\varpi^n).$$ Alors $\HC = \varprojlim_n \HC_n$ et chaque 
   $\HC_n$ est fini.




\subsection{Géométrie de l'espace symétrique de Drinfeld}
\label{paragraphhdkrectein}

Nous allons maintenant décrire l'espace symétrique de Drinfeld $\H_K^d$. Il s'agit de l'espace  analytique sur $K$ dont les $C$-points sont \[\H_K^d (C)=\P^d (C)\backslash \uni{H}{H\in \HC}{}.\]




     On dispose d'une application $G$-équivariante \[\tau : \H^d_K(C)\to \{\text{normes sur } K^{d+1}\}/\{\text{homothéties}\} \] donnée par $$\tau (z): v\mapsto |\som{z_i v_i}{i=0}{d}|$$ si $z=[z_0,\cdots, z_d] \in \H_K^d (C)$. L'image $\tau(z)$ ne dépend pas du représentant de $z$ car les normes sont vues à  homothétie près. Le fait de prendre le complémentaire des hyperplans $K$-rationnels assure que $\tau(z)$ est séparée et donc une norme sur $K^{d+1}$. D’après un résultat classique de Iwahori-Goldmann \cite{iwgo}, l'espace des normes sur $K^{d+1}$ à  homothétie près s'identifie bijectivement (et de manière $G$-équivariante) à  l'espace topologique $|\BC\TC|$, ce qui permet de voir $\tau$ comme une application 
$$\tau: \H_K^d(C)\to |\BC\TC|.$$ 

  Nous renvoyons à \cite[§I.4]{boca} quand $d=1$ ou \cite[§2.1]{wa} pour la justification des faits suivants. 
Soit $\sigma\in\BC\TC_{k}$ un simplexe de type $(e_0, e_1,\cdots, e_k)$ et posons 
 $$\H_{K,\sigma}^d:=\tau^{-1} (\sigma), \,\, \H_{K,\mathring{\sigma}}^d:=\tau^{-1} (\mathring{\sigma}).$$
 L'ouvert $\H_{K,\sigma}^d$ est un affinoïde, admettant un modèle entier $\H_{\OC_K,\sigma}^d =\spf  (\hat{A}_\sigma)$ où $\hat{A}_\sigma$ est le complété $p$-adique de 
$\OC_K [X_0,\cdots, X_d,\frac{1}{P_\sigma}]/(\pro{X_{d_{i}}}{i=0}{k}-\varpi),$
pour un certain polynôme $P_{\sigma}\in \OC_K[X_0,...,X_d]$ qui est décrit dans \cite[§2.1]{wa} ou dans \cite[§1.5]{J3}. Nous ne servirons ici que du cas où $\sigma=s$ est un sommet. L'algèbre $\hat{A}_s$ est alors le complété $p$-adique de \begin{equation}\label{eq:hatas}
\OC_K [X_0,\cdots, X_{d-1},\frac{1}{\prod_{a=(a_i)_i\in \F^{d+1}\backslash \{0\}} (\tilde{a}_0 X_0+\cdots \tilde{a}_{d-1}X_{d-1}+ \tilde{a}_d)}]
\end{equation}
avec $\tilde{a}=(\tilde{a}_i)_i $ un relevé de $a$ dans $\OC_K$.
 Pour $\H_{K,\mathring{\sigma}}^d$, il s'agit de l'ouvert \[\{z\in\P^d_K : \forall 0\le i\le k, \forall a,b\in M_i\backslash M_{i-1},\forall c\in M_{i-1},|\left\langle c,z\right\rangle|<|\left\langle a,z\right\rangle|=|\left\langle b,z\right\rangle|\]
avec $M_{-1}=\varpi M_k$. Considérons les ouverts
\[C_r= \{x=(x_1,\cdots, x_r)\in \B^r_K|\, \forall a\in \OC^{r+1}_K\backslash \varpi \OC^{r+1}_K, \  1=|\left\langle (1, x),a \right\rangle| \},\]
\[A_k=\{y=(y_0,\cdots, y_{k-1})\in \B^k_K|\,  1>|y_{k-1}|>\cdots>|y_0|>|\varpi|\}\]
et les morphismes 
$$ \H_{K,\mathring{\sigma}}^d \rightarrow C_{e_i -1}, \,\, 
  [z_0,\cdots, z_d]  \mapsto  (\frac{ z_{d_{i}+1} }{ z_{d_{i}} },\frac{ z_{d_{i}+2} }{ z_{d_{i}} },\cdots, \frac{ z_{d_{i+1}-1} }{ z_{d_{i}} })\et $$
 $$\H_{K,\mathring{\sigma}}^d \rightarrow A_k,\,\, 
  [z_0,\cdots, z_d]  \mapsto  (\frac{ z_{d_{0}} }{ z_{d} },\frac{ z_{d_{1}} }{ z_{d} },\cdots, \frac{ z_{d_{k-1}} }{ z_{d} }).$$
 Il est montré dans \cite[6.4]{ds}  que   
    les morphismes ci-dessus induisent un isomorphisme
\[ \H_{K,\mathring{\sigma}}^d \cong A_k \times \pro{C_{e_i-1}}{i=0}{k}\cong A_k \times C_\sigma.\]

\begin{rem}\label{remniinvci}

Nous avons introduit précédemment une décomposition en somme adapté au simplexe :
\[M_k =N_0\oplus \cdots\oplus N_k\]
avec $N_i=\left\langle f_{d_{i-1}+1},\cdots, f_{d_{i}} \right\rangle$. En particulier, un vecteur unimodulaire $a$ est dans $M_i$ si et seulement si la projection sur $ N_{i+1} \oplus \cdots\oplus N_k$ est divisible par $\varpi$. Ainsi tout vecteur $a$ de $M_i$ peut s'écrire sous la forme $a=a_1+a_2$ avec $a_1\in N_i$, $a_2\in M_{i-1}$ et on a d'après la description des morphismes ci-dessus :
\[\frac{l_{a_1}}{z_{d_i}}\in \Of^*(C_{e_i-1}).\]
\end{rem}

\subsection{Géométrie de la fibre spéciale de $\H_{\OC_K}^d$ \label{paragraphhdf}}

Si $s$ est un sommet de $\BC\TC$ on note 
$$\ost(s)=\uni{\mathring{\sigma}}{\sigma\ni s }{},\,\, \fst(s)= \uni{\sigma}{\sigma\ni s }{}$$
 l'étoile ouverte, respectivement fermée de $s$. Pour un simplexe $\sigma$ on note 
 $$\ost(\sigma)=\inter{\ost(s)}{s\in \sigma}{},\,\, \fst(\sigma)=\inter{\fst(s)}{s\in \sigma}{}.$$
  
  Les composantes irréductibles de la fibre spéciale de $\H_{\OC_K}^d$ 
sont indexées par l'ensemble des sommets de $\BC\TC_0$: pour chaque sommet $s$ la composante correspondante est $\H_{\F,\ost(s)}^d$.
On obtient ainsi un recouvrement admissible $(\H_{K,\ost(s)}^d)_{s\in \BC\TC_0}$
 de $\H_K^d$ dont les intersections d'ouverts sont de la forme $\H_{K,\ost(\sigma)}^d$ pour $\sigma\subset\BC\TC$ un simplexe. En fibre spéciale, le lieu lisse de cette intersection est l'ouvert \[\H_{\F,\ost(\sigma)}^d\backslash \uni{\H_{\F,\ost(s)}^d}{s\notin \sigma}{}=\H_{\F,\mathring{\sigma}}^d\] car $\mathring{\sigma}=\sigma \backslash \uni{\sigma'}{\sigma'\subsetneq\sigma}{}$. 
En particulier, le lieu lisse d'une composante irréductible est  \[\H_{\F,s}^d\cong\P_{ \F}^d\backslash \uni{H}{H}{}\]
 o\`u $H$ parcourt l'ensemble des hyperplans $\F$-rationnels (cf \eqref{eq:hatas}). 
De plus, $\H_{\F,\ost(s)}^d$ est une compactification qui s’obtient comme suit (\cite[sous-section III.1.]{gen}  ou \cite[ 4.1.2]{wa2}). Posons $Y_0 =\P_{\F}^d$ et construisons
par éclatements successifs une suite d'espaces \[Y_d\to Y_{d-1}\to \cdots \to Y_0.\]
Supposons que l’on ait préalablement construit $Y_0 , \cdots , Y_i$. On a des morphismes $p_i :Y_i \to Y_{i-1}$ et $\tilde{p}_i =p_i\circ\cdots\circ p_1 : Y_i \to Y_{0}$. On pose $Z_i$ le transformé strict par $\tilde{p}_i$ de l’union des espaces de
codimension $i + 1$ dans $Y_0$. On définit $Y_{i+1}$ comme l’éclaté de $Y_i$ suivant $Z_i$. Alors, on a $Y_d\cong\H_{\F,\ost(s)}^d$.

\subsection{Interprétation modulaire de l'espace de Drinfeld}

   Pour construire le premier revêtement de l'espace de Drinfeld, nous avons besoin d'une interprétation modulaire de cet espace, ce qui demande quelques notions et notations. 

  Rappelons que $D$ est une $K$-algèbre centrale à division, de dimension $(d+1)^2$ et d'invariant $\frac{1}{d+1}$ et considérons $\OC_{(d+1)}$ l'anneau des entiers d'une extension non-ramifiée de $K$, de degré $d+1$ contenue dans $D$. Si $A$ est une $\OC_K$-algèbre, un \emph{$\OC_D$-module formel} sur ${\rm Spec}(A)$ (ou, plus simplement, sur $A$) est un groupe formel $X$ sur $A$ muni d'une action de $\OC_D$, notée $\iota : \OC_D \to \mathrm{End}(X)$, qui est compatible avec l'action naturelle de $\OC_K$ sur l'espace tangent $\mathrm{Lie}(F)$, \emph{i.e.} pour $a$ dans $\OC_K$, $d\iota (a)$ est la multiplication par $a$ dans $\mathrm{Lie}(F)$. Le 
  $\OC_D$-module formel $X$ est dit \emph{spécial} si $\mathrm{Lie}(X)$ est un 
  $\OC_{(d+1)}\otimes_{\OC_K} A$-module localement libre de rang $1$. On a le r\'esultat classique suivant:




\begin{prop}
  Sur un corps alg\'ebriquement clos de caract\'eristique $p$, il existe, à isogénie près, un unique 
   $\OC_D$-module formel sp\'ecial de dimension $d+1$ et de $(\OC_K$-)hauteur $(d+1)^2$. 
\end{prop}

  On notera $\Phi_{\bar{\F}}$ l'unique (à isogénie près) 
$\OC_D$-module formel sp\'ecial $\Phi_{\bar{\F}}$ sur $\bar{\F}$ de dimension $d+1$ et hauteur $(d+1)^2$ (l'entier $d$ étant fixé par la suite, nous ne le faisons pas apparaître dans la notation $\Phi_{\bar{\F}}$). 

   Soit 
   $\mathrm{Nilp}_{\OC_{\breve{K}}}$ la catégorie des $\OC_{\breve{K}}$-algèbres sur lesquelles 
   $\varpi$ est nilpotent. 
Consid\'erons le foncteur $\GC^{Dr} : \mathrm{Nilp}_{\OC_{\breve{K}}} \to \mathrm{Ens}$ envoyant $A\in \mathrm{Nilp}_{\OC_{\breve{K}}}$ sur l'ensemble des classes d'isomorphisme de triplets $(\psi, X, \rho)$ avec : 
\begin{itemize}[label= \textbullet] 
\item $\psi : \bar{\F} \to A/ \varpi A$ est un $\F$-morphisme, 
\item $X$ est un $\OC_D$-module formel sp\'ecial de dimension $d+1$ et de hauteur $(d+1)^2$ sur $A$, 
\item $\rho : \Phi_{\bar{\F}} \otimes_{\bar{\F}, \psi} A/ \varpi A \to X_{A/ \varpi A}$ est une quasi-isog\'enie de hauteur z\'ero. 
\end{itemize}
   
 Le théorème fondamental suivant, à la base de toute la théorie, est dû à Drinfeld : 

\begin{theo}[\cite{dr2}]\label{Drrep}
Le foncteur $\GC^{Dr}$ est repr\'esentable par $\H_{\OC_{\breve{K}}}^d$.
\end{theo} 
 

\begin{rem}
On d\'efinit le foncteur $\tilde{\GC}^{Dr}$ de la m\^eme manière que $\GC^{Dr}$ mais en ne fixant plus la hauteur de la quasi-isog\'enie $\rho$. Alors $\tilde{\GC}^{Dr}$ est, lui aussi, repr\'esentable par un schéma formel $\widehat{\MC}^0_{Dr}$ sur $ \spf (\OC_K)$, qui se  d\'ecompose $$\tilde{\GC}^{Dr}= \coprod_{h \in \Z} \GC^{Dr,(h)},$$ o\`u $\GC^{Dr,(h)}$ est d\'efini comme pr\'ec\'edemment en imposant que la quasi-isog\'enie $\rho$ soit de hauteur $(d+1)h$. Chacun des $\GC^{Dr,(h)}$ est alors isomorphe (non canoniquement) au foncteur $\GC^{Dr}$, ce qui induit un isomorphisme non-canonique $$\widehat{\MC}^0_{Dr}\cong \H^d_{\OC_{\breve{K}}}\times \Z.$$ 
     
\end{rem}

\subsection{La tour de Drinfeld}

 On note $\mathfrak{X}$ le $\OC_D$-module formel sp\'ecial universel sur $\H_{\OC_{\breve{K}}}^d$ (cf. th. \ref{Drrep}) et 
$\tilde{\XG}$ le module formel spécial universel déduit de la représentabilité de $\tilde{\GC}^{Dr}$. Pour tout entier $n\geq 1$, l'action de $\Pi_D^n$ induit une isogénie de 
 $\XG$ et de $\tilde{\XG}$. Le schéma en groupes 
$\mathfrak{X}[ \Pi_D^n] = \ker( \mathfrak{X} \xrightarrow{\Pi_D^n} \mathfrak{X})$ (resp. $\tilde{\XG}[\Pi_D^n]$)
est fini plat, de rang $q^{n(d+1)}$ sur $\H_{\OC_{\breve{K}}}^d$ (resp. $\widehat{\MC}^0_{Dr}$). 

  On note $\Sigma^0=\H^d_{\breve{K}}$ et $\MC^0_{Dr}=(\widehat{\MC}^0_{Dr})^{\rm rig}\cong\H^d_{\breve{K}}\times \Z$. Pour $n\geq 1$ on 
 d\'efinit 
\[ \Sigma^n := (\mathfrak{X}[\Pi_D^n] \backslash \mathfrak{X}[ \Pi_D^{n-1} ])^{\text{rig}},\  \MC^n_{Dr}:= (\tilde{\XG}[\Pi^n_D]\backslash\tilde{\XG}[\Pi_D^{n-1}])^{\rm rig}.\]
Les  morphismes d'oubli $\Sigma^n \to \Sigma^0$ et $\MC^n_{Dr}\to\MC^0_{Dr}$ d\'efinissent des rev\^etements finis \'etales de groupe de Galois\footnote{De même, les morphismes intermédiaires $\Sigma^n\to\Sigma^{n-1}$ et $\MC^n_{Dr}\to\MC^{n-1}_{Dr}$ sont des revêtements finis étales de groupes de Galois $(1+\Pi^{n-1}_D \OC_D)/(1+\Pi^n_D \OC_D)$. Les tours obtenues définissent aussi des revêtements pro-étales de groupe de Galois $\OC_D^*$} $\OC_D^{*}/(1+ \Pi^n_D \OC_D)$. On a encore des isomorphismes non-canoniques $\MC^n_{Dr}\cong \Sigma^n\times \Z$ et les revêtements respectent ces décompositions.

  Le groupe $G$ s'identifie au groupe des quasi-isogénies de $\XG$, il agit donc naturellement sur chaque niveau de la tour $(\MC^n_{Dr})_{n\geq 0}$. De même, le groupe $\OC_D^{*}$ permute les points de $\Pi^n_D $-torsion et $\OC_D^{*}$ agit sur $\MC^n_{Dr}$ à travers son quotient $\OC_D^{*}/(1+ \Pi^n_D \OC_D)\simeq {\rm Gal}(\MC^n_{Dr}/\MC^0_{Dr})$. Ces deux actions commutent entre elle et les revêtements $\MC^n_{Dr}\to\MC^0_{Dr}$ sont $G$-équivariants. En revanche, le revêtement $\Sigma^n\to\Sigma^{0}$ est seulement  $\gln_{d+1}(\OC_K)$-équivariant\footnote{En fait, il l'est pour le groupe qui préserve $\Sigma^1\cong \Sigma^1\times \{0\}$ à savoir $v\circ \det^{-1}((d+1)\Z)\subset G$.}.

\subsection{Le premier revêtement} \label{paragraphrevsigman}

Nous nous intéressons désormais au cas $n=1$. On peut encore définir une flèche de réduction $\nu$ de $\Sigma^1$ vers l'immeuble de Bruhat-Tits $\BC\TC$ s'inscrivant dans le diagramme :
\[
\xymatrix{ 
\Sigma^1  \ar[d]^{\pi} \ar[dr]^{\nu} & \\
 \H^d_{\breve{K}} \ar[r]^-{\tau} & \BC \TC }.\]
Pour tout sous-complexe simplicial $T \subseteq \BC \TC$, on note 
$$\Sigma^1_{T}= \nu^{-1}(T),\,\,\Sigma^1_{L,T}= \Sigma^1_{T}\otimes_{\breve{K}} L.$$

Le groupe de Galois  de $\Sigma^1$ est $\F_{q^{d+1}}^*$. Ce groupe est cyclique et son  cardinal
 $$N=q^{d+1}-1$$ est premier \`a $p$. C'est un revêtement modérément ramifié et ces deux propriétés joueront un rôle central dans la suite. 
  Le schéma 
 $\mathfrak{X}[ \Pi_D]$ est, en particulier, un sch\'ema en $\F_p$-espaces vectoriels et la condition que $\mathfrak{X}$ soit sp\'ecial entraîne que $\mathfrak{X}[ \Pi_D]$ est un sch\'ema de Raynaud. Pour énoncer les conséquences de cette observation, introduisons quelques notations.
 
\'Ecrivons  $$\tilde{u}_1(z)=(-1)^{d}\prod_{a \in (\F)^{d+1} \backslash \{ 0 \}}a_0z_0 +\cdots +a_d z_d\in \Of(\A^{d+1}_{\F}\backslash\{0\}).$$ Si on fixe $b\in (\F)^{d+1} \backslash \{ 0 \}$, on construit $u_1(z)=(b_0z_0 +\cdots +b_d z_d)^{-N}\tilde{u}_1$ une fonction inversible de $\H_{\F,s}^d\cong\P_{ \F}^d\backslash \bigcup_{H\in \HC_1} H$ pour  $s$ le sommet standard de $\BC\TC$. Notons que la projection de $u_1$ dans $\Of^*(\H_{\F,s}^d)/ (\Of^*(\H_{\F,s}^d))^N$ ne dépends pas du choix de $b$ et celui-ci n'aura donc pas d'incidence sur les résultats à suivre. Pour simplifier, nous pourrons prendre $b=(0,\cdots,0,1)$. On peut aussi relever $u_1$ en une fonction inversible dans $\Of^*(\H_{\OC_{\breve{K}},s}^d)$ voire même dans $\Of^*(\H^d_{\breve{K},\ost(s)})$ que l'on notera encore $u_1$. 
 
  En utilisant la classification des sch\'emas de Raynaud, on a d'après \cite[Théorème 4.9 et Remarque 4.10]{J3}\footnote{Notons que l'espace $\mathring{U}_{1,\breve{K}}$ dans \cite[Remarque 4.10]{J3} coïncide avec $\H^d_{\breve{K},\ost(s)}$.} (cet énoncé étend les résultats de Teitelbaum \cite[Théorème 5]{teit2} pour $d=1$ et de Wang  \cite[Lemme 2.3.7.]{wa} pour $d$ quelconque) :

\begin{theo}\label{theoeqsigma1}
Il existe $u\in \Of^*(\H^d_{\breve{K}})$ vérifiant $u\equiv u_1 \pmod{\Of^*(\H^d_{\breve{K},\ost(s)})^N}$ telle que \[\Sigma^1\cong \H^d_{\breve{K}}((\varpi u)^{\frac{1}{N}}).\] En particulier, $\Sigma^1_{\ost(s)}\cong \H^d_{\breve{K},\ost(s)}((\varpi u_1)^{\frac{1}{N}})$.

\end{theo}

\begin{rem}\label{remactglngen}

Notons que le résultat précédent ne décrit pas l'action de $\gln_{d+1}(\OC_K)\varpi^{\Z}=\stab_{G}(\ost(s))$. Toutefois, d'après \cite[Remarque 4.12]{J3}, toutes les actions possibles sur $\Sigma^1$ commutant avec le revêtement se déduisent l'une de l'autre en tordant par un caractère \[\chi \in \homm (\gln_{d+1}(\OC_K),\mu_N(\H^d_{\breve{K},\ost(s)}))\cong\F^*.\] Nous déterminerons l'action provenant de l'interprétation modulaire dans \ref{lemlieulisse}.

\end{rem}

\section{Cohomologie des vari\'et\'es de Deligne-Lusztig\label{ssectiondl}}

\subsection{Vari\'et\'es de Deligne-Lusztig}\label{sssectiondldef}

  Considérons les groupes algébriques
$G=\gln_{d+1, \bar{\F}}$ et $G_0=\gln_{d+1,\F}$, ainsi que le morphisme de Frobenius $F$ défini par $(a_{i,j})_{i,j} \mapsto (a_{i,j}^q)_{i,j}$. 
Soit $B$ le sous-groupe des matrices triangulaires sup\'erieures, $T$ le tore des matrices diagonales, $U$ le sous-groupe de $B$ des matrices unipotentes. 

On identifie le groupe de Weyl 
$W=N_G(T)/T$ à 
 $\SG_{d+1}$ par le biais des matrices de permutation. 
 Soit $w$ la matrice de permutation associ\'ee au cycle $(0, 1, \dots, d) \in \SG_{d+1}$. On définit 
\[ Y(w):= \{ gU \in G/U |\,\, g^{-1}F(g) \in U w U \} \et \]
\[ X(w):= \{ gB \in G/B |\,\, g^{-1}F(g) \in B w B \}. \]

Il existe $\pi$ rendant le diagramme suivant commutatif : 
\[ \xymatrix{
Y(w) \ar[r]^-{\iota} \ar[d]^-{\pi} & G/U \ar[d] \\
X(w) \ar[r]^-{\iota} & G/B }. \]

Le groupe $\gln_{d+1}(\F)=G(\bar{\F})^{F=1}$ agit sur $Y(w)$ et $X(w)$ par multiplication à gauche. Le groupe fini commutatif $$T^{wF}:= \{ t \in T |\,\, wF(t)w^{-1} =t \}$$ agit librement (par multiplication à droite) sur $Y(w)$. La fl\`eche $\pi$ est un rev\^etement fini \'etale et induit un isomorphisme $\gln_{d+1}(\F)$-équivariant $$Y(w)/T^{wF} \xrightarrow{\sim} X(w).$$

  On peut rendre ces objets plus explicites comme suit \cite[2.2]{DL}. D'une part
 $T^{wF}$ s'identifie \`a $\F_{q^{d+1}}^*$ via l'application $x \in \F_{q^{d+1}}^* \mapsto {\rm diag}(x, Fx, \dots, F^dx)$. D'autre part, considérons la variété 
 $$\Omega^d_{\F}:=\P_{\F}^d \backslash \bigcup_H H,$$ o\`u $H$ parcourt l'ensemble des hyperplans $\F$-rationnels. Elle possède une action naturelle de $\gln_{d+1}(\F)$.  
 Soit $e_0=(1,0,...,0)$ et $H_0= \ker(l_{e_0})$. On rappelle que l'on a aussi construit dans la section précédente deux applications $\tilde{u}_1(z)\in \Of(\A^{d+1}_{\F}\backslash\{0\})$ et $u_1(z)=\tilde{u}_1(z)/z_d^N\in\Of^*(\Omega^d_{\F})$.  


\begin{prop}\label{propdleq}
On a des identifications $\gln_{d+1}(\F)$-équivariantes entre $X(w)$ et $\Omega^d_{\F}$ et entre $Y(w)$ et $$ \{z \in \A_{\F}^{d+1} \backslash \{ 0 \} :\tilde{u}_1(z)=1\}=:\dl^d_{\F}.$$ De plus, $\pi$ est induite par  la projection naturelle $\A_{\F}^{d+1} \backslash \{ 0 \}\to \P_{\F}^{d}$

\end{prop}

\begin{rem}\label{remdlkumm}

\begin{itemize}
\item En envoyant $(z_0,...,z_d)\in Y(w)$ sur $(z_d^{-1}, [z_0:\cdots :z_d])$ et $(t,z=[z_0:\cdots: z_d])\in \Omega^d_{\F}(u_1^{1/N})$ sur 
$(t^{-1}, \frac{z_1}{tz_d},\cdots,\frac{z_d}{tz_d})$ sous l'identification précédente, on obtient un isomorphisme $\F_{q^{d+1}}^*$-équivariant (le deuxième terme n'a pas d'action naturelle de  $\gln_{d+1}(\F)$...)
$$Y(w)\simeq  \Omega^d_{\F}(u_1^{1/N}).$$ 
\item Comme dans la remarque \ref{remactglngen}, toutes les actions possibles de $\gln_{d+1}(\F)$ sur $ \Omega^d_{\F}(u_1^{1/N})$ commutant avec le revêtement se déduisent l'une de l'autre en tordant par un caractère \[\chi \in \homm (\gln_{d+1}(\OC_K),\mu_N(\Omega^d_{\F}))\cong\F^*,\] et le lemme \ref{lemlieulisse} exhibe une identification naturelle avec l'ensemble des actions considérées dans  \ref{remactglngen}. Dans  \ref{lemlieulisse}, nous verrons que l'action provenant de l'interprétation modulaire de $\Sigma^1$ et l'action naturelle sur $\dl^d_{\F}$ coïncident sous cette bijection.
\end{itemize}
\end{rem}

\begin{proof} On identifie $G/B$ à la variété des drapeaux complets de 
$(\bar{\F})^{d+1}$. On vérifie facilement qu'un 
 drapeau $\{ 0 \} \subsetneq D_0 \subsetneq \dots \subsetneq D_{d}=(\overline{\F})^{d+1} $
est dans $X(w)$ si et seulement si pour tout $i$ on a $$D_i=D_0 \oplus FD_0 \oplus \dots \oplus F^{i} D_0.$$ On obtient un plongement  $X(w) \to \P_{\F}^{d}, (D_i) \mapsto D_0$. La projection d'un point $z=(z_0, \dots , z_d) \in (\bar{\F})^{d+1} \backslash \{ 0 \}$ est dans l'image de ce morphisme si et seulement si $(z, Fz, \dots , F^d z)$ est une base de $(\bar{\F})^{d+1}$, ce qui revient à dire que $\det((z_i^{q^j})_{0\le i,j\le d})$ est non nul. Mais 
 \[  \det((z_i^{q^j})_{i,j})^{q-1}= \prod_{a \in (\F)^{d+1} \backslash \{ 0 \}} l_{a}(z) =(-1)^{d}\tilde{u}_1(z).\]
 On en d\'eduit alors un isomorphisme $$X(w) \xrightarrow{\sim} \P_{\F}^d \backslash \bigcup_{H \in \P^d(\F)} H= \Omega^d_{\F}.
$$ 

La variété $ G/U$ classifie les paires $((D_i)_i, (e_i)_i)$ avec $(D_i)_i$ un drapeau et $e_i\in D_i\backslash D_{i-1}$. Une paire $((D_i)_i, (e_i)_i)$ est dans $Y(w)$ si et seulement si $(D_i)_i\in X(w)$ et \[\forall i< d, F^{i}e_0\equiv e_{i} \pmod{D_{i-1}} \et F^{d+1}e_0\equiv e_{0} \pmod{{\rm Vect}(e_1,\cdots,e_d)}.\] Ainsi la flèche $((D_i)_i, (e_i)_i)\mapsto e_0$ induit un  plongement $Y(w)\to\A_{\F}^{d+1} \backslash \{ 0 \}$ rendant le diagramme suivant commutatif : 
\[ \xymatrix{
Y(w) \ar[r]^-{\iota} \ar[d] & \A_{ \F}^{d+1} \backslash \{ 0 \} \ar[d] \\ 
X(w) \ar[r]^-{\iota} & \P_{ \F}^d
}. \] Un point $x=(z_0, \dots , z_d) \in \A_{ \F}^{d+1} \backslash \{ 0 \}$ est dans l'image de ce morphisme si et seulement si $\det((z_i^{q^j})_{0\le i,j\le d})=(-1)^d \det(F\cdot (z_i^{q^j})_{0\le i,j\le d})$.  Cela revient à écrire  \[\iota: Y(w)\xrightarrow{\sim} \{z \in \A_{\F}^{d+1} \backslash \{ 0 \} :\tilde{u}_1 (z)=1\}. \]

\end{proof}

\subsection{Cohomologie étale des vari\'et\'es de Deligne-Lusztig \label{sssectioncohodl}}

On note $\dl^d_{\bar{\F}}$ l'extension des scalaires de $\dl_{\F}^d$ \`a $\bar{\F}$. Soit $l\ne p$ un nombre premier, nous allons rappeler la description de la partie cuspidale  
de la cohomologie $l$-adique à support compact de $\dl^d_{\bar{\F}}$. 

 
   Soit $\theta: \F_{q^{d+1}}^*\to \bar{\Q}_l^*$ un caractère, on dit que  $\theta$ est 
    \emph{primitif} 
  s'il ne se factorise pas par la norme $\F_{q^{d+1}}^*\to \F_{q^e}^*$ pour tout diviseur propre $e$ de $d+1$.  Si $M$ est un $  \overline{\mathbb{Q}}_l[\F_{q^{d+1}}^*]$-module on note 
 $$M[\theta]={\rm Hom}_{\F_{q^{d+1}}^*}(\theta, M).$$

 Si $\pi$ est une représentation de $\gln_{d+1}( \F)$, on dit que $\pi$ est \emph{cuspidale} si 
 $\pi^{N(\F)}=0$ pour tout radical unipotent $N$ d'un parabolique propre de $\gln_{d+1}$.  
  La théorie de Deligne-Lusztig (ou celle de Green dans notre cas particulier) fournit: 
 
 \begin{theo}\label{DLet}
  Soit $\theta: \F_{q^{d+1}}^*\to \bar{\Q}_l^*$ un caractère.
  
  a) Si $\theta$  est primitif, alors 
  $\hetc{i}(\dl_{\bar{\F}}^d, \bar{\Q}_l)$ est nul pour $i\ne d$ et 
  $$\bar{\pi}_{\theta,l}:=\hetc{d}(\dl_{\bar{\F}}^d, \bar{\Q}_l)[\theta]$$
  est une $\gln_{d+1}( \F)$-représentation irréductible, cuspidale, de dimension  $(q-1)(q^2-1) \dots (q^d-1)$. Toutes les repr\'esentations cuspidales sont ainsi obtenues.

  b) Si $\theta$ n'est pas primitif, aucune repr\'esentation cuspidale n'intervient dans $\bigoplus_{i}\hetc{i}(\dl_{\bar{\F}}^d, \bar{\Q}_l)[\theta]$. 
 \end{theo}

  \begin{proof}
   Voir \cite[cor. 6.3]{DL}, \cite[th. 7.3]{DL}, \cite[prop. 7.4]{DL}, \cite[prop. 8.3]{DL}, \cite[cor. 9.9]{DL},  pour ces résultats classiques. 
  \end{proof}
  
    Ainsi, la partie cuspidale $\hetc{*}(\dl_{\bar{\F}}^d, \bar{\Q}_l)_{\rm cusp}$ de $\hetc{*}(\dl_{\bar{\F}}^d, \bar{\Q}_l)$ est concentrée en degré $d$, où elle est donnée par $\bigoplus_{\theta} \bar{\pi}_{\theta,l}\otimes\theta$, la somme directe portant sur tous les caractères primitifs.
 
 \begin{rem} (voir \cite[6.3]{DL} et \cite[Proposition 6.8.(ii) et remarques]{yosh}) Soit $N=q^{d+1}-1$ et fixons un isomorphisme $\F_{q^{d+1}}^*\simeq \Z/N\Z$ et 
 $\Z/N\Z^{\vee}\simeq \Z/N\Z$.
    Soient $\theta_{j_1}$ et $\theta_{j_2}$ deux caract\`eres primitifs vus comme des \'el\'ements de $\Z/N\Z$ via $j_1$, $j_2$, les repr\'esentations $\bar{\pi}_{\theta_{j_1}}$ et  $\bar{\pi}_{\theta_{j_2}}$ sont isomorphes si et seulement si il existe un entier $n$ tel que $j_1= q^n j_2$ dans $\Z/N\Z$. 
 \end{rem}



\subsection{Cohomologie rigide des variétés de Deligne-Lusztig}

   Nous aurons besoin d'un analogue des résultats présentés dans le paragraphe précédent pour 
   la cohomologie rigide. Cela a été fait par Grosse-Klönne dans \cite{GK5}. 
   Si $\theta: \F^*_{q^{d+1}}\to \bar{K}^*$ est un caractère, posons 
   $$\bar{\pi}_{\theta}=\hrigc{*}(\dl_{\F}^d/ \bar{K})[ \theta]:=\bigoplus_{i}\hrigc{i}(\dl_{\F}^d/ \bar{K})[ \theta],$$
 où 
 $$\hrigc{i}(\dl_{\F}^d/ \bar{K}):=\hrigc{i}(\dl_{\F}^d)\otimes_{W(\F)[1/p]} \bar{K}$$
 et où $M[\theta]$ désigne comme avant la composante $\theta$-isotypique de $M$.

\begin{theo}\label{theodlpith}
Fixons un premier $l\ne p$ et un isomorphisme $\bar{K} \cong \bar{\Q}_l$. Si $\theta$ est un caract\`ere primitif, alors  $$\bar{\pi}_{\theta}:=\hrigc{d}(\dl_{\F}^d/ \bar{K})[ \theta]$$ est isomorphe en tant que 
$\gln_{d+1}(\F)$-module à $\bar{\pi}_{\theta,l}$, en particulier c'est une représentation irréductible cuspidale.


\end{theo}

\begin{proof} Cela se fait en trois étapes, cf. \cite[4.5]{GK5}. Dans un premier temps, on montre \cite[3.1]{GK5} que 
les $\bar{K}[\gln_{d+1}(\F) \times \F^*_{q^{d+1}}]$-modules virtuels 
\[ \sum_i (-1)^i \hetc{i}(\dl_{\bar{\F}}^d, \bar{\Q}_l) \et \sum_i (-1)^i \hrigc{i}(\dl_{\F}^d/ \bar{K}) \]
co\"incident. 
Il s'agit d'une comparaison standard des formules des traces de Lefschetz en
cohomologies étale $l$-adique et rigide. Dans un deuxième temps (et c'est bien la partie délicate du résultat),
on montre que $\bigoplus_{i}\hrigc{i}(\dl_{\F}^d/ \bar{K})[ \theta]$ est bien concentré en degré $d$, cf. \cite[th. 2.3]{GK5}. On peut alors conclure en utilisant le théorème \ref{DLet}.
\end{proof}



\section{Cohomologie de de Rham et revêtements cycliques modérés}

   Dans ce chapitre $L$ sera une extension non ramifiée de $K$, donc $\varpi$ en est une uniformisante. Toutes les cohomologies de de Rham seront calculées sur le complexe surconvergeant ie. nous écrirons par abus $\hdr{*} (X)$ pour tout espace analytique  $X$ au lieu de $\hdr{*} (X^{\dagger})$.

\subsection{Réduction semi-stable généralisée}

    Soit $\XC$ un schéma formel topologiquement de type fini sur $\spf  (\OC_L)$, de fibre g\'en\'erique $\XC_\eta$ et de fibre sp\'eciale $\XC_s$. On a une flèche de spécialisation $\spg: \XC_\eta \rightarrow \XC_s$. Pour tout sous-schéma $Z\subset \XC_s$ on note $]Z[_{\XC}$ le tube de $Z$ dans $\XC_\eta$, i.e. l'espace analytique $$]Z[_{\XC}=\spg^{-1} (Z)\subset \XC_\eta.$$  

 On dit que $\XC$ est de réduction semi-stable généralisée s'il existe un recouvrement ouvert (Zariski) $\XC=\uni{U_t}{t\in T}{}$ et un jeu de morphismes étales (pour certains $r\leq d$ et $\alpha_i\geq 1$) \[\varphi_t : U_t\to  \spf (\OC_L\left\langle x_1,\cdots , x_d\right\rangle/(x_1^{\alpha_1}\cdots x_r^{\alpha_r}-\varpi )).\]  Dans ce cas, quitte à rétrécir les ouverts $U_t$ et à prendre $r$ minimal, on peut supposer que les composantes irréductibles de la fibre spéciale $\bar{U}_t$  de $U_t$ sont les $V(\bar{x}^*_i)$ pour $i\leq r$ avec $\bar{x}^*_i =\bar{\varphi}_t (\bar{x}_i)$. Elles ont les multiplicités $\alpha_i$. On dit que $\XC$ est de réduction  semi-stable si de plus tous les $\alpha_i$ valent $1$. Dans ce cas, les fibres spéciales $\bar{U}_t$ sont réduites.

 \subsection{Enoncé du résultat principal}

Soit $\XC$ un schéma formel sur $\spf  (\OC_L)$, de réduction semi-stable généralisée, de 
 fibre g\'en\'erique $\XC_\eta$ et de fibre sp\'eciale $\XC_s$. On note $(Y_i)_{i\in I}$ l'ensemble des composantes irréductibles de 
 $\XC_s$. On suppose que le recouvrement $\XC_s=\uni{Y_i}{i\in I}{}$ est localement fini, i.e. pour toute partie finie $J$ de $I$, les composantes $Y_j$ pour $j\in J$ n'intersectent qu'un nombre fini de composantes irréductibles de $\XC_s$. Si $J$ est un sous-ensemble de $I$, on note 
 $$Y_J=\inter{Y_j}{j\in J}{}.$$

Le but de cette section est de prouver le théorème suivant: 

\begin{theo}\label{theoprinc}

Soient $\XC$ semi-stable  généralisé et $(Y_i)_{i\in I}$ comme ci-dessus, et soit $\pi: \TC\rightarrow \XC_\eta$ un rev\^etement \'etale  de groupe de Galois $\mu_n$ avec $n$ premier à $p$. Pour toute partie finie $J$ de $I$ la flèche de restriction induit un isomorphisme \[\hdr{*} (\pi^{-1}(]Y_J[_\XC))\fln{\sim}{} \hdr{*} (\pi^{-1}(]Y_{J}\backslash \bigcup_{i\notin J}Y_i[_\XC)).\]
si $\XC$ est de réduction semi-stable (non généralisée) ou $n=1$ (i.e. $\TC=\XC_\eta$). 
\end{theo}

\begin{rem}

\begin{enumerate}

\item Si $\XC$ est de réduction semi-stable (non généralisée) et $n=1$ (i.e. $\TC=\XC_\eta$) le théorème ci-dessus a été démontré par 
Grosse-Klönne \cite[Theorem 2.4.]{GK2}. Il s'agit d'un point crucial dans sa preuve de la finitude de la cohomologie rigide. 
   Le principal intérêt de notre généralisation est la présence du revêtement cyclique $\pi$ de la fibre générique de $\XC_\eta$.

\item Si de plus $\XC$ est algébrisable et $|J|=1$, le résultat \cite[Lemme 5.6]{zhe} est un analogue en cohomologie étale $l$-adique du théorème ci-dessus.  
\end{enumerate}

\end{rem}

 Comme dans la preuve originale, on procède en deux étapes. 
  On applique dans un premier temps un certain nombre de réductions assez techniques 
 (cf. \ref{paragraphred1} et \ref{paragraphred2}) pour se ramener à l’étude des revêtements de couronnes. Ces étapes sont 
 similaires à la démonstration de Grosse-Klönne, qui utilise des recouvrements bien choisis et la suite spectrale de Cech. Dans notre cas, on reprend  les mêmes recouvrements de $\XC_\eta$ puis on les tire en arrière par $\pi$ pour étudier l’espace $\TC$. Le seul point technique à adapter dans ces réductions est la vérification que l'espace final obtenu est bien décrit par un revêtement de couronnes (voir \ref{lemprod}). 

 La deuxième étape de la preuve (et la plus technique) est le calcul 
 de la cohomologie de de  Rham d'un revêtement cyclique modéré d'une couronne. Cela se fait par des calculs directs
 sur le complexe de de Rham, et fournit une description très explicite de ces groupes de cohomologie. Pour énoncer le résultat nous avons besoin de quelques notations. Soient $s_1,...,s_d,r_1,...,r_d\in |\bar{K}^*|$ tels que $s_i\leq r_i$ pour tout $i$ et considérons la couronne 
 $$X=\A^d_{rig, L}(\frac{x_i}{r_i},\frac{s_i}{x_i})_{1\le i\le d}=\{(x_1,...,x_d)\in\A^d_{rig, L} |\, s_i\leq |x_i|\leq r_i\}.$$
 Soit $n$ un entier premier à $p$ et considérons le revêtement de Kummer $\TC=X((\lambda x^\beta)^{1/n})$   avec $\lambda\in L^*$ et $\beta\in\Z^d$. Cela est loisible, puisque nous allons voir que tout revêtement $\pi : \TC\to X$ galoisien cyclique d'ordre $n$ est de cette forme. On dispose donc sur $\TC$ d'une racine $n$-ième $t$ de $\lambda x^\beta$. On définit enfin 
  $$\pi_0 ={\rm PGCD}(n,\beta_1, \cdots, \beta_d), \, \tilde{n}=\frac{n}{\pi_0}, \,\tilde{\beta}=\frac{\beta}{\pi_0}, \, t_0=\frac{t^{\tilde{n}}}{x^{\tilde{\beta}}}.$$
  Enfin, si $q\ge 1$ et $I=\left\lbrace i_1<\cdots < i_q\right\rbrace$ on pose 
  $$d\log (x_I)=d\log(x_{i_1})\wedge\cdots\wedge d\log(x_{i_q}).$$
  
\begin{theo}
 Avec les notations ci-dessus, on a des isomorphismes naturels $$\hdr{q}(X)\simeq \drt{L\cdot d\log (x_I)}{I}{}$$ et
 \[\hdr{q} (\TC)\simeq \drt{t_0^i\drt{L\cdot d\log (x_I)}{I}{}}{i=0}{\pi_0 -1} \cong\drt{t_0^i\hdr{q}(X)}{i=0}{\pi_0 -1}\]
où $I$  parcours les parties de $\left\llbracket 1,d\right\rrbracket$ de cardinal $q$ dans les sommes précédentes.

\end{theo}

\begin{rem} 
\begin{enumerate}

\item On déduit facilement du théorème que 
si  $X'\subset X$ sont deux couronnes et si $\TC\to X$, $\TC'\to X'$ sont deux revêtements compatibles (ie. $\TC' =\TC\times_X X'$) alors  la fl\`eche de restriction $\hdr{*} (\TC)\rightarrow \hdr{*} ({\TC'})$ est un isomorphisme qui respecte la décomposition en  parties isotypiques.

\item En fait, tous ces résultats sont vrais pour une classe plus générale  d’espaces, que l’on appellera tores monômiaux. Nous aurons besoin de ce degré de généralité et nous renvoyons à \ref{proph1etmonom},  \ref{theodrcouronne} et \ref{coroinclqsiso} pour les énoncés dans ce cadre.
\end{enumerate}

\end{rem}

\subsection{Tores monômiaux et leurs revêtements cycliques modérés\label{sssectionrevmonom}}

\begin{defi}

On appellera tore mon\^omial de dimension $d$ un $L$-espace analytique $X$ de la forme\footnote{On a utilisé les notations multi-indice standard, par exemple $x^{\alpha}=x_1^{\alpha_1}\cdots x_d^{\alpha_d}$.} \[X=\{x=(x_1,\cdots, x_d) \in \A^d_{rig, L} : s_i\le |x_i|\le r_i \et \rho\le |x^\alpha|\le \mu\}\] pour $s_i\leq r_i\in |\bar{K}^*|$, $\alpha=(\alpha_1,...,\alpha_d)\in\N^d$ et $\rho\le \mu\in [s^{\alpha}, r^{\alpha}]\cap 
|\bar{K}^*|$). 

\end{defi}

On appelle tore mon\^omial semi-ouvert un espace défini par les mêmes inégalités qu'un tore mon\^omial, mais potentiellement strictes. 
 Nous souhaitons étendre un r\'esultat de Berkovich  \cite[Lemma 3.3]{ber6} au cas des tores monômiaux. 
 
\begin{prop}\label{proph1etmonom}
Soit $X$ un tore mon\^omial de dimension $d$ (semi-ouvert), $S$ un espace $K$-analytique, $n$ un entier premier à $p$. La projection canonique $\varphi : X_S:=X\times S \to S$ induit un isomorphisme
 \[R^q \varphi_* \mu_n\simeq\mu_n (-q)^{\binom{d}{q}}.\]

\end{prop}

\begin{proof}

Soit $X$ un tore monômial et $s,r,\alpha,\rho,\mu$ les données associées. 
Nous allons montrer le résultat par  récurrence sur la dimension $d$. 
Si $d=1$, tous les tores monômiaux sont des couronnes qui ont été traitées dans \cite[Lemma 3.3]{ber6} (on peut aussi appliquer \cite[Lemme 4.4.]{J1} puis la suite exacte de Kummer). 

Soit  $d>1$, en projetant sur les $d-1$ premières coordonnées, on obtient un morphisme $\psi : X\to Y$ vers le tore monômial :
\[Y=\{x=(x_1,\cdots, x_{d-1}) \in \A^{d-1}_{rig, L} : s_i\le |x_i|\le r_i \et  \rho r_d^{-\alpha_d}\le |x^\alpha|\le \mu s_d^{-\alpha_d}\}\]  
avec $\bar{\alpha}=(\alpha_1 ,\cdots , \alpha_{d-1}) $. 
Soient $u: Y_S\to S$ 
les projections naturelles, alors 
$u\circ \psi=\varphi$, 
donc 
$$ R \varphi_* \mu_n\simeq Ru_*R\psi_*\mu_n.$$
Par hypothèse de récurrence 
et la suite spectrale de Leray il suffit d'établir 
les isomorphismes  \[ R^q \psi_* \mu_n=\begin{cases}  \mu_n &  {\rm si}\ q=0\\ \mu_n (-1) &  {\rm si}\ q=1\\ 0 & {\rm si}\ q>1\end{cases}.\]

  Notons que  $R^q\psi_* \mu_n$ est un faisceau surconvergent   (puisque les faisceaux constants le sont et que cette propriété est stable par image directe et twist à la Tate), on peut donc tester les isomorphismes ci-dessus fibre à fibre. Les tiges 
  du faisceau 
  $R^q\psi_* \mu_n$ se calculent grâce au théorème de changement de base \cite[TH 3.7.3]{djvdp} et font intervenir la cohomologie du faisceau $\mu_n$ sur les fibres de $\psi$. Ces fibres sont des couronnes de dimension $1$ (sur le corps de définition du point considéré), et on a déjà vu le calcul de ces groupes de cohomologie, ce qui permet de conclure. 

Pour le cas semi-ouvert, on peut trouver un recouvrement croissant de  $X$ par des tores monômiaux fermés $X_k$. Pour $\varphi_k :X_k \times S\to S$ la projection sur le second facteur, on a d'après la discussion précédente un système projectif constant de complexe $(\rrr\varphi_{k,*}\mu_n)_k$ et on en déduit le calcul de $\rrr\varphi_{*}\mu_n$ du cas fermé.


  
\end{proof}

\begin{rem}\label{remkummrev}

Si $X$ est un tore monômial sur un corps complet $S=\spg(L)$ et une couronne $Y$ qui le contient, alors, en reprenant le raisonnement par récurrence précédent sur $Y$, on montre la bijectivité du morphisme naturel de restriction $\het{1}(Y,\mu_n)\iso\het{1}(X,\mu_n) $. Par suite exacte de Kummer sur $Y$ (voir \cite[th. 3.25]{vdp} pour l'annulation du groupe de Picard de $Y$ et \cite[Lemme 4.4.]{J1} pour le calcul des fonctions inversibles), \[\het{1}(X,\mu_n)\cong L^*/(L^*)^n\times\pro{(x_i^\Z/x_i^{n\Z})}{i\le d}{}.\] 

En particulier, tout rev\^etement étale de groupe de Galois $\mu_n$ de $X_L$ est un rev\^etement de Kummer de la forme $X((\lambda x^\beta)^{1/n})$ pour $\beta$ dans $\Z^d$ et $\lambda$ dans $L^*$.


\end{rem}

\subsection{Cohomologie de de Rham d'un revêtement cyclique modéré d'un tore monomial\label{sssectiondrmonom}}

  Le but de ce paragraphe est de calculer la cohomologie de de Rham d'un revêtement 
  cyclique $\TC = X((\lambda x^\beta)^{1/n})$ (avec $\beta\in \Z^d$ et $\lambda\in L^*$) d'un tore monômial \[X=\{x=(x_1,\cdots, x_d) \in \A^d_{rig, L} : s_i\le |x_i|\le r_i \et \rho\le |x^\alpha|\le\mu\}.\]
Posons $$\pi_0 ={\rm PGCD}(n,\beta_1, \cdots, \beta_d), \, \tilde{n}=\frac{n}{\pi_0}, \, \tilde{\beta}=\frac{\beta}{\pi_0}, \, t_0=\frac{t^{\tilde{n}}}{x^{\tilde{\beta}}}.$$

\begin{theo}\label{theodrcouronne}
On dispose d'isomorphismes naturels  
$$\hdr{q}(X)=\drt{L\cdot d\log (x_I)}{I\subset \left\llbracket 1,d\right\rrbracket \\ |I|=q}{}$$
et 
$$\hdr{q} (\TC)\cong\drt{t_0^i\hdr{q}(X)}{i=0}{\pi_0 -1}.$$ 

\end{theo}

\begin{rem} Pour comprendre l'enonc\'e du th\'eor\`eme, il est int\'eressant d'\'etudier le cas analytique complexe. Si l'on prend un espace $X$ de $\C^d$ défini par les mêmes in\'egalités qu'un tore monômial i.e. \[X=\{x=(x_1,\cdots, x_d) \in \C^d : s_i\le |x_i|\le r_i \et \rho\le |x^\alpha| \le \mu\},\]  alors $X$ a le type d'homotopie d'un tore. La cohomologie de de Rham est donc donn\'ee par (Künneth) :
  \[ \hdr{q}(X)=\underset{|I|=q}{\bigoplus}  \C \cdot d\log( x_I). \]
 D'apr\`es la correspondance de Galois entre les rev\^etements et les sous-groupes de $\pi_1(X)$, un rev\^etement cyclique $\TC$ de $X$ a le type d'homotopie d'une union disjointe de tores. Cette union s'\'ecrit :
\[ \{ (x, t)\in \C^d/ \Z^d\times \C || t^n=x^\beta     \}.  \]
  Le nombre  de  composantes  connexes est la constante $\pi_0$ introduite dans l'\'enonc\'e du th\'eor\`eme. Comme $t_0^{\pi_0}=1$, la famille $\{t_0^i\}_i$ engendre le même $\C$-espace vectoriel que l'ensemble des idempotents pour les différentes composantes connexes, et on obtient
\[  \hdr{q}(\TC)= \underset{0\le i\le \pi_0-1}{\bigoplus}t_0^i \hdr{q}(X)    .\]
\end{rem}

 Avant de passer à la preuve, mentionnons quelques conséquences utiles:

\begin{coro}\label{coroinclqsiso}

On reprend les notations pr\'ec\'edentes et on se donne un autre tore mon\^omial $X'$ inclus dans $X$. Si $\TC'$ est la restriction de $\TC$ \`a $X'$ i.e. $\TC' =\TC\times_X X'$, alors la fl\`eche de restriction $\hdr{*} (\TC)\rightarrow \hdr{*} ({\TC'})$ est un isomorphisme qui respecte la décomposition en parties isotypiques.

\end{coro}

\begin{proof}
La base explicite du th\'eor\`eme \ref{theodrcouronne} est conserv\'ee par la restriction $\Omega_{\TC^\dagger/L}^q \to \Omega_{(\TC')^\dagger /L}^q$
d'où la bijectivité. Pour l'assertion sur les parties isotypiques, l'inclusion induit une application $\mu_n$-équivariante entre les complexes de de Rham et le résultat s'en déduit. Pour un argument plus explicite, on a la décomposition en espaces propres $\Omega^q_{\TC^\dagger/L}=\drt{t^i \Omega^q_{X^\dagger/L}}{i=0}{n-1}$ et pour $i$ fixé chaque $t_0^i\hdr{q}(X)$ est un espace propre de $\hdr{q} (\TC)$. 

\end{proof}

\begin{coro}\label{corosouv}

Les conclusions de \ref{theodrcouronne} et \ref{coroinclqsiso} sont encore vraies quand $X$ est un tore mon\^omial semi-ouvert.

\end{coro}

\begin{proof}
On écrit $X=\bigcup_i X_i$ comme une réunion croissante de tores mon\^omiaux. On a alors $\TC=\bigcup_i \TC\times_X X_i = \bigcup_i \TC_i$. Fixons $i_0\in \N$, d'après ce qui précède, on a \[\limp_i\hdr{*} (\TC_i)[\chi]=\hdr{*} (\TC_{i_0})[\chi]\et\rrr^1\limp_i\hdr{*} (\TC_i)[\chi]=0\] d'où $\hdr{*} (\TC)[\chi]\cong\hdr{*} (\TC_{i_0})[\chi]$.
\end{proof}

Passons à la preuve du résultat.
Nous commen\c{c}ons par traiter le cas des tores monomiaux i.e. $\pi= \id$ et $\TC=X$. Comme dans la d\'efinition, on se donne $r, s, \alpha, \rho, \mu$ d\'efinissant $X$ et on choisit des constantes $u=(u_i)_i$, $v=(v_i)_i$ et $w_1$, $w_2$ dans $L$ telles que 
\[ |u_i|=r_i,\ |v_i|=s_i,\ |w_1|= \rho,\ |w_2| = \mu. \]
Si $x=(x_i)_i$ d\'esigne la variable sur $X$, alors 
\[ X= \spg(L \langle \frac{x}{u}, \frac{v}{x}, \frac{x^{\alpha}}{w_1}, \frac{w_2}{x^{\alpha}} \rangle). \]
En particulier, toute fonction $f$ appartient \`a $\Of(X^{\dagger})$ admet un d\'eveloppement unique $f= \sum_{\nu \in \Z^d} a_{\nu} x^{\nu}$. Nous aurons besoin du résultat technique suivant: 
 
\begin{lem}[Int\'egration]\label{propint}
Pour $i \le d$ et $f= \sum_{\nu \in \Z^d} a_{\nu} x^{\nu} \in \Of(X^{\dagger})$, il existe une section surconvergente de d\'eveloppement $\sum_{\nu \in \Z^d :\nu_i\neq 0} \frac{1}{\nu_i}' a_{\nu} x^{\nu} \in \Of(X^{\dagger})$. 
\end{lem}   

\begin{proof}
R\'e\'ecrivons $f$ sous la forme 
\[ \sum_{ (\beta, \gamma, \delta_1, \delta_2) \in \N^d \times \N^d \times \N \times \N} m_{\beta, \gamma, \delta_1, \delta_2} (\frac{x}{u})^{\beta} (\frac{v}{x})^{\gamma} (\frac{x^{\alpha}}{w_1})^{\delta_1} (\frac{w_2}{x^{\alpha}})^{\delta_2} \]
tel que il existe $h >1$ tel que $h^{e(\beta, \gamma, \delta_1, \delta_2)} | m_{\beta, \gamma, \delta_1, \delta_2} | \to 0$ avec $e(\beta, \gamma, \delta_1, \delta_2)= \sum_{1 \le k \le d} \beta_k + \sum_{1 \le k \le d} \gamma_k + \delta_1 + \delta_2$. Nous voulons montrer que le développement suivant défini bien une section surconvergente : \[\sum_{ (\beta, \gamma, \delta_1, \delta_2) } \frac{m_{\beta, \gamma, \delta_1, \delta_2}}{\beta_i - \gamma_i + \alpha_i(\delta_1- \delta_2)}  (\frac{x}{u})^{\beta} (\frac{v}{x})^{\gamma} (\frac{x^{\alpha}}{w_1})^{\delta_1} (\frac{w_2}{x^{\alpha}})^{\delta_2}\] où $(\beta, \gamma, \delta_1, \delta_2)$ parcourt les termes tels que $\beta_i - \gamma_i + \alpha_i(\delta_1- \delta_2)\neq 0$.

Pour tout $h_1 \in ] 1, h[$, on a $(\frac{h_1}{h})^{e(\beta, \gamma, \delta_1, \delta_2)} | \frac{1}{ \beta_i - \gamma_i + \alpha_i (\delta_1- \delta_2)} | \to 0$ pour le filtre des parties finies  et donc 
\[ h_1^{e(\beta, \gamma, \delta_1, \delta_2)} | \frac{m_{\beta, \gamma, \delta_1, \delta_2}}{\beta_i - \gamma_i + \alpha_i(\delta_1- \delta_2)}   | = h^{e(\beta, \gamma, \delta_1, \delta_2)} | m_{\beta, \gamma, \delta_1, \delta_2} | \frac{(h_1/h)^{e(\beta, \gamma, \delta_1, \delta_2)}}{| \beta_i- \gamma_i+ \alpha_i(\delta_1- \delta_2)| }  \to 0  \]  
et on obtient bien une section surconvergente sur $X$ de d\'eveloppement $\sum_{\nu:\nu_i\neq 0} \frac{1}{\nu_i} a_{\nu} x^\nu$. 
\end{proof}

Une $q$-forme surconvergente admet un unique d\'eveloppement $\omega=\sum_{\nu \in \N^d, I \subset \llbracket 1, d \rrbracket} a_{\nu, I}x^{\nu} d\log(x_I)$. Nous dirons que $\omega$ contient un terme en $x_i$ (resp.  un terme en $d\log(x_i)$) s'il existe $a_{\nu, I} \neq 0$ avec $\nu_i \neq 0$ (resp. avec $i \in I$). On appelle $\Omega^q_{X^{\dagger}/L} [r] $ le sous-module des formes qui ne contiennent aucun terme en $x_i$ ou $d \log(x_i)$ pour $i >r$. On observe l'inclusion $d(\Omega^q_{X^{\dagger}/L} [r] ) \subset \Omega^{q+1}_{X^{\dagger}/L} [r] $.   

On impose l'ordre lexicographique sur les couples $(q,r)$ et on montre par r\'ecurrence sur $(q,r)$ l'\'egalit\'e suivante \footnote{On a posé $\Omega^{q-1}_{X^{\dagger}/L} [r]=0$ si $q=0$} : 
\[ (\Omega^q_{X^{\dagger}/L} [r])^{d=0} = d(\Omega^{q-1}_{X^{\dagger}/L} [r] )\oplus  \bigoplus_{\substack{ I \subset \llbracket 1,r \rrbracket \\ |I | =q}}  L d\log(x_I). \] Il est aisé de voir que les modules apparaissant dans le terme de droite sont en somme directe et nous laissons la vérification de ce fait au lecteur. Nous allons seulement prouver que ces modules engendrent bien le sous-ensemble des $q$-formes fermées.  

Si $q=0$, comme $X$ est g\'eom\'etriquement connexe, on a   $\hdr{0}(X)= L= L \cdot d\log(x_{\emptyset})$. 

Soit $q \ge 1$, supposons le r\'esultat vrai pour tout $(q',r') < (q,r)$. Si $\omega \in (\Omega^q_{X^{\dagger}/L}[r])^{d=0}$, elle se d\'ecompose de mani\`ere unique de la forme \footnote{si $r=1$, on pose $\Omega^{q}_{X^{\dagger}/L}[r-1]=L$}
\[ \omega = \sum_{j \in \Z} x_r^j  \omega_j^{(0)}+ \sum_{j \in \Z}  x_r^j \omega_j^{(1)} \wedge d\log(x_r) \]
avec  $\omega_j^{(i)} \in \Omega^{q-i}_{X^{\dagger}/L}[r-1]$. Par fermeture de $\omega$, 
\begin{equation}\label{eqdom}
d\omega= \sum_{j \in \Z} x_r^j d \omega_j^{(0)} + \sum_{j \in \Z} x_r^j ((-1)^q j\omega_j^{(0)} + d \omega_j^{(1)}) \wedge d \log(x_r)  =0
\end{equation}
d'o\`u $d\omega_j^{(0)} =0$ et $d\omega_j^{(1)}= (-1)^{q-1} j \omega_j^{(0)}$ toujours par unicité de la décomposition. En particulier,  $d\omega_0^{(1)}=d\omega_0^{(0)}=0$.

D'apr\`es \ref{propint}, la somme suivante est une $(q-1)$-forme surconvergente de $X$, 
\[ \eta= \sum_{j \neq 0} (-1)^{q-1} \frac{1}{j} x_r^j \omega_j^{(1)}. \]
On v\'erifie par calcul direct, $\omega- d\eta= \omega_0^{(0)} + \omega_0^{(1)} \wedge d\log(x_r)$ (d'après \eqref{eqdom}).  On a montré que les formes $\omega_0^{(0)}$ et $\omega_0^{(1)}$ étaient fermés. On peut leur appliquer l'hypothèse de récurrence, ce qui permet de conclure. 

On s'intéresse maintenant au cas général. Nous cherchons à calculer la cohomologie d'un revêtement sur $X$ de la forme 
\[ \TC= X((\lambda x^{\beta})^{1/n})= \{ (x,t) \in X \times \A_{rig,L}^1 : t^n=\lambda x^{\beta} \}. \]
Quitte \`a \'etendre $L$, on suppose qu'il contient les racines $\pi_0$-i\`emes de l'unit\'e et que $\lambda$ vaut $1$. On a alors une d\'ecomposition 
\[ \TC= \coprod_{\zeta \in \mu_{\pi_0}(L)} X( \zeta x^{\tilde{\beta}}):= \coprod_{\zeta \in \mu_{\pi_0}(L)} \TC_{\zeta}. \]   
Appelons $\Lf_{\zeta}$ le polyn\^ome interpolateur de Lagrange s'annulant sur $\mu_{\pi_0}(L) \setminus \{ \zeta \}$ et valant $1$ en $\zeta$. Prenons $t \in \Of^*(\TC)$ une racine $n$-i\`eme de $x^{\beta}$ et $t_0= \frac{t^{\tilde{n}}}{x^{\tilde{\beta}}}$. Alors $\Lf_{\zeta}(t_0)$ est l'idempotent associ\'e \`a $\TC_{\zeta}$. Supposons que $\TC_{\zeta} \to X$ induise un isomorphisme $\hdr{*}(X) \cong \hdr{*}(\TC_{\zeta})$, on obtient une suite d'isomorphismes 
\[ \hdr{*}(\TC) = \bigoplus_{\zeta} \Lf_{\zeta}(t_0) \hdr{*}(\TC_{\zeta}) \cong \sum_{j=0}^{\pi_0-1} t_0^j \hdr{*}(X) \]
car $\{ t_0^j \}_j$ et $\{ \Lf_{\zeta}(t_0) \}_{\zeta}$ engendrent le m\^eme $L$-espace vectoriel. Il suffit ainsi de raisonner sur chaque $\TC_{\zeta}$ i.e. on peut supposer $\pi_0=1$.

Consid\'erons maintenant le rev\^etement de groupe de Galois $\mu_n^d$ suivant : 
\[ \tilde{\TC} = X(x_1^{1/n})(x_2^{1/n}) \cdots(x_d^{1/n})= \{ (x,t_1, \dots, t_d) \in X\times \A^d_{rig,L} : t_i^n= x_i \}. \] 
Les fl\`eches $((x_1, \dots, x_d), t_1, \dots, t_d) \mapsto (t_1, \dots, t_d)$ et $(t_1, \dots, t_d) \mapsto ((t_1^n , t_2^n,\dots, t_d^n), t_1, \dots, t_d)$ induisent une bijection 
\[ \tilde{\TC} \cong \{ (t_1, \dots, t_d) \in \A^d_{rig,L} :   s_i^{1/n}\le |t_i|\le r_i^{1/n} \et  \rho^{1/n}\le |t_1^{\alpha_1} t_2^{\alpha_2} \dots t_d^{\alpha_d}| \le  \mu^{1/n} \}. \]
On en d\'eduit la suite d'\'egalit\'es :
\[ \hdr{i}(\tilde{\TC})= \bigoplus_{|I|= i } L d\log(t_I) = \bigoplus_{|I|=i} L d\log(x_I) = \hdr{i}(X) \]
On remarque ais\'ement que $\tilde{\TC}$ est un rev\^etement de $\TC$ de groupe de Galois ab\'elien $\fix_{\mu_n^d}(\TC)=\{(\gamma_i)_i\in \mu_n^d : \sum_i \gamma_i\beta_i=0\}$. La fl\`eche naturelle $\hdr{*}(\TC) \to \hdr{*}(\tilde{\TC})$ identifie $\hdr{*}(\TC)$ \`a $\hdr{*}(\tilde{\TC})^{\fix_{\mu_n^d}(\TC)}$ et est donc injective. On obtient un diagramme commutatif : 
\[ \xymatrix{
\hdr{*}(\tilde{\TC}) & \\
\hdr{*}(\TC) \ar@{^{(}->}[u] & \ar[l] \hdr{*}(X) \ar[lu]_{\sim} }.\]
Ainsi, $\hdr{*}(\TC) \cong \hdr{*}(X)$.

\subsection{Une première réduction \label{paragraphred1}}

   Revenons maintenant au contexte du théorème \ref{theoprinc}. En particulier, on dispose du schéma formel $\XC$ semi-stable généralisé, d'un revêtement cyclique $\pi : \TC\to \XC_\eta$ d'ordre premier à $p$ et on note $(Y_i)_{i\in I}$ les composantes irréductibles de sa fibre spéciale. Rappelons que l'on note $Y_J=\bigcap_{j\in J} Y_j$ pour $J\subset I$. Notons que les deux résultats qui vont suivre seront valables dans les deux cas considérés dans \ref{theoprinc}. Nous spécialiserons au cas $\pi={\rm Id}$ ou au cas $\XC$ semi-stable (non généralisé) dans la section  \ref{ssectionfin}.

\begin{lem}\label{lemred1}
Pour démontrer le théorème \ref{theoprinc}
il suffit de prouver la bijectivité de 
\[\hdr{*} (\pi^{-1}(]Y_J[_\XC))\fln{\sim}{}\hdr{*} (\pi^{-1}(]Y_J\backslash Y_I[_\XC))\] quand $\XC=\spf (A)$ est affine formel, connexe et possède un morphisme étale \[\varphi: \spf (A)\rightarrow \spf (\OC_L \left\langle x_1,\cdots,x_d\right\rangle /(x_1^{\alpha_1}\cdots x_r^{\alpha_r} -\varpi)).\] 

\end{lem}

\begin{proof}
On considère uniquement l'ensemble (fini) des composantes de $I$ qui intersectent $J$, ce qui nous permet de supposer que $I$ est fini.
Pour $J\subset  I$ et $M\subset J^c$, on note $$\mathring{Y}_{J,M} =Y_J\backslash \uni{Y_l}{l\in M}{},$$
i.e. les points dans toutes les composantes irréductibles de $J$ qui évitent celles de $M$. Si de plus $
m$ est un entier positif, on note $$\mathring{Y}_{J,M}^{m}=\uni{\mathring{Y}_{J,N}}{N\subset M\\ |N|=m}{},$$ i.e. les points dans toutes les composantes de $J$ qui évitent au moins $m$ composantes dans $M$. On remarque les égalités $Y_J = \mathring{Y}_{J,M}^{0}$ et $\mathring{Y}_{J,M}^{|M|}=\mathring{Y}_{J,M}$.

On raisonne sur la chaîne d'inclusions
\[{Y}_{J}\supset \mathring{Y}_{J,J^c}^{1}\supset \cdots \supset \mathring{Y}_{J,J^c},\]
 il suffit donc de montrer  (pour tout $m$) la bijectivité de la flèche naturelle
\[\hdr{*} (\pi^{-1}(]\mathring{Y}_{J,J^c}^{m}[_\XC))\fln{\sim}{} \hdr{*} (\pi^{-1}(]\mathring{Y}_{J,J^c}^{m+1}[_\XC)).\]
Notons que $\mathring{Y}^1_{J,J^c}=Y_J\backslash Y_I$. Nous allons construire un schéma formel auxiliaire $\tilde{\XC}$ de réduction semi-stable généralisée pour déduire le cas $m$ quelconque à partir du cas $m=0$. 

Chaque $]\mathring{Y}_{J,J^c}^{m}[$ admet par définition un recouvrement admissible par les ouverts suivants $(]\mathring{Y}_{J,N}[)_{N\subset J^c : |N|=m}$. Les intersections finies de ces ouverts vérifient\footnote{On suppose que tout élément de $Q$ est de cardinal $m$.}
\[\inter{\mathring{Y}_{J,N}}{N\in Q}{}=\mathring{Y}_{J,M},\]
avec $M=\bigcup_{N\in Q} N$. Lorsque $Q$ varie, $M$ parcourt l'ensemble des parties disjointes de $J$ pour tout $M$ de cardinal au moins $m$. On se ramène donc, grâce à la suite spectrale de Cech, à montrer l'isomorphisme \[\hdr{*} (\pi^{-1}(]\mathring{Y}_{J,M}\cap\mathring{Y}_{J,J^c}^{m}[_\XC)) =\hdr{*} (\pi^{-1}(]\mathring{Y}_{J,M}[_\XC))\fln{\sim}{} \hdr{*} (\pi^{-1}(]\mathring{Y}_{J,M}\cap\mathring{Y}_{J,J^c}^{m+1}[_\XC))\] pour tout $M$ de cardinal au moins $m$. Quand $M$ n'est pas de cardinal $m$, on a $\mathring{Y}_{J,M}=\mathring{Y}_{J,M}\cap\mathring{Y}_{J,J^c}^{m+1}$ et la bijectivité est triviale.

Quand $M$ est  de cardinal $m$, on se place dans la préimage  $\tilde{\XC}\subset \XC$ de l'ouvert en fibre spéciale $\XC_s\backslash \left(\uni{Y_i}{i\in M}{}\right)$ par la flèche de spécialisation. Les composantes irréductibles de $\tilde{\XC}_s$ sont indexées par $I\backslash M$. On reprend les notations $\tilde{Y}_J$, $\mathring{\tilde{Y}}_{J,N}$ et $\mathring{\tilde{Y}}_{J,N}^m$ pour $J\subset M^c$ et $N\subset  (M\cup J)^c$. On observe 
\[\mathring{Y}_{J,M}=\tilde{Y}_{J}\et \mathring{Y}_{J,M}\cap\mathring{Y}_{J,J^c}^{m+1}=\mathring{\tilde{Y}}_{J,(J\cup M)^c}^{1}.\]
On s'est ramené à la condition suffisante du lemme par l'observation $\tilde{Y}_J\backslash \tilde{Y}_I=\mathring{\tilde{Y}}_{J,J^c}^1$.

Il reste à expliquer pourquoi on peut supposer  $\XC$  affine formel. 
On se donne un recouvrement affine $\XC=\uni{U_s}{s\in S}{}$ et un jeu de morphismes étales \[\varphi_s : U_s\to  \spf (\OC_L\left\langle x_1,\cdots , x_d\right\rangle/(x_1^{\alpha_1}\cdots x_r^{\alpha_r}-\varpi )).\] On note $U_T=\inter{U_t}{t\in T}{}$ pour $T\subset S$. On a un recouvrement admissible $\TC=\uni{\pi^{-1} (U_{s,\eta})}{s\in S}{}$ donc par une nouvelle application de la suite spectrale de Cech, il suffit de prouver la bijectivité de \[\hdr{*} (\pi^{-1}(]Y_J[_\XC \cap U_{T,\eta}))\fln{\sim}{}\hdr{*} (\pi^{-1}(]Y_J\backslash Y_I[_\XC\cap U_{T,\eta})).\] pour tout $T\subset S$. On s'est donc ramené \`a $\XC=U_T =\spf (A)$ affine formel, $\TC=\pi^{-1} (U_{T,\eta})$ affinoïde et $\varphi={\varphi_s}_{|U_T}: \XC\to  \spf (\OC_L\left\langle x_1,\cdots , x_d\right\rangle/(x_1^{\alpha_1}\cdots x_r^{\alpha_r}-\varpi ))$ pour un certain $s\in T$. 
Quitte à réduire $\XC$, on peut supposer qu'il est connexe. 
\end{proof}

\subsection{Une deuxième réduction\label{paragraphred2}}

   Nous supposons maintenant que nous sommes dans le contexte du lemme \ref{lemred1}. En particulier,  $\XC=\spf (A)$ est affine 
   et connexe et possède un morphisme étale $\varphi: \spf (A)\rightarrow \spf (\OC_L \left\langle x_1,\cdots,x_d\right\rangle /(x_1^{\alpha_1}\cdots x_r^{\alpha_r} -\varpi))$. 
Notons $x_i^* =\varphi^* (x_i)\in A$ et $\bar{x}_i$ son image en fibre spéciale. Quitte à réduire $\XC$, on peut supposer $I= \left\llbracket 1,r\right\rrbracket$, $J=\left\llbracket 1,|J|\right\rrbracket \subset I$, $Y_i =V(\bar{x}_i)$. Notons aussi que 
$\TC$ est affinoide, disons $\TC=\spg(B)$.

\begin{lem}\label{lemred2}

 Pour montrer \ref{lemred1}, il suffit de prouver que pour toute partition $\left\llbracket 1,r\right\rrbracket =J_1 \cup J_2 \cup  J_3$ et tous $\lambda,\beta\in ]0,1[\cap |\bar{K}^*|$, les restrictions de $\TC$ à $\tilde{C}$ et à $\tilde{C}'$ ont la même cohomologie de de Rham, où   \[\tilde{C}=\left\lbrace s\in \XC_\eta :  |x_{j_1}^* (s)|\leq \lambda, |x_{j_2}^* (s)|=\beta, |x_{j_3}^* (s)|\leq\beta, \, \forall j_k\in J_k\right\rbrace,\]\[ \tilde{C}'=\left\lbrace s\in \XC_\eta :  |x_{j_1}^* (s)|\leq \lambda, |x_{j_2}^* (s)|=\beta, |x_{j_3}^* (s)|=\beta, \, \forall j_k\in J_k\right\rbrace.\] 

\end{lem}

\begin{proof}

Rappelons que : \[]Y_J[_\XC=\{s\in \XC_\eta |\,  \forall j\in J\,, |x_j^*(s)| < 1\},\]
\[]Y_J\backslash Y_I[_\XC=\{s\in ]Y_J[_\XC|\, \exists  i\in J^c, |x_i^*(s)| = 1\}.\] 
Nous allons filtrer ces espaces par des ouverts plus simples. Pour 
$\lambda\in ]0,1[\cap |\bar{K}^*|$ posons\footnote{Pour tout $M\subset I$ nous confondrons $M$ et $\{x_i^* :j\in M\}$.}
\[\XC_\eta (\frac{J}{\lambda})=\{s\in \XC_\eta|\, \forall j\in J, |x_j^*(s)|\le \lambda\}.\]
On définit $]Y_J\backslash Y_I[_\XC (\frac{J}{\lambda})$ d'une manière semblable 
et on remarque que $\pi^{-1}(]Y_J[_\XC)=\uni{\pi^{-1}(\XC_\eta (\frac{J}{\lambda})}{\lambda< 1}{})$ est un recouvrement admissible (et la situation est similaire pour $\pi^{-1}(]Y_J\backslash Y_I[_\XC)$). Il suffit donc de montrer la bijectivité de 
\[\hdr{*} (\pi^{-1}(\XC_\eta(\frac{J}{\lambda})))\fln{\sim}{}\hdr{*} (\pi^{-1}(]Y_J\backslash Y_I[_\XC (\frac{J}{\lambda})))\] pour tout $\lambda$ comme ci-dessus. 

 Ensuite, intéressons-nous à  l'espace 
\[\uni{\XC_\eta (\frac{J}{\lambda},\frac{\beta}{\left\lbrace i \right\rbrace})}{i\in J^c}{}=\{s\in \XC_\eta (\frac{J}{\lambda})|\,\exists  i\in J^c,\,\, |x_i^*(s)|\ge \beta\}\] pour $\beta\in ]0,1[\cap |\bar{K}^*|$.
 On remarque la suite d'inclusion $$\pi^{-1}(]Y_J\backslash Y_I[_\XC (\frac{J}{\lambda})) \subset \pi^{-1}(\uni{\XC_\eta (\frac{J}{\lambda},\frac{\beta}{\left\lbrace i \right\rbrace})}{i\in J^c}{}) \subset \pi^{-1}(\XC_\eta(\frac{J}{\lambda})).$$ \'Etudions d'abord la flèche en cohomologie induite par la première inclusion. On a un recouvrement admissible $(\pi^{-1}(\XC_\eta (\frac{J}{\lambda},\frac{\beta}{\left\lbrace i \right\rbrace})))_{i\in J^c}$ de $\pi^{-1}(\uni{\XC_\eta (\frac{J}{\lambda},\frac{\beta}{\left\lbrace i \right\rbrace})}{i\in J^c}{})$ et les intersections sont de la forme $\pi^{-1}(\XC_\eta (\frac{J}{\lambda},\frac{\beta}{M}))$ pour $M\subset J^c$. 
 
 De plus, on a l'identité \[\pi^{-1}(]Y_J \backslash Y_I[_\XC  (\frac{J}{\lambda} ))^\dagger=\bigcap_{\beta < 1}\  \pi^{-1}(\XC_\eta (\frac{J}{\lambda},\frac{\beta}{M}))^\dagger\] car chaque voisinage stricte de $\pi^{-1}(]Y_J \backslash Y_I[_\XC  (\frac{J}{\lambda} ))$ contient un ouvert de la forme  $\pi^{-1}(\XC_\eta (\frac{J}{\lambda},\frac{\beta}{M}))$ pour un certain $\beta$. En comparant les complexes de de Rham surconvergeant, on obtient la bijectivité de
 \[\varinjlim_{\beta < 1}\  \hdr{*} ( \pi^{-1}(\XC_\eta (\frac{J}{\lambda},\frac{\beta}{M})))\fln{\sim}{}\hdr{*} (\pi^{-1}(]Y_J\backslash Y_I[_\XC (\frac{J}{\lambda})))\]
 puis celle de \[\varinjlim_{\beta < 1}\  \hdr{*} ( \pi^{-1}(\uni{\XC_\eta (\frac{J}{\lambda},\frac{\beta}{\left\lbrace i \right\rbrace})}{i\in J^c}{}))\fln{\sim}{}\hdr{*} (\pi^{-1}(]Y_J\backslash Y_I[_\XC (\frac{J}{\lambda})))\] 
 grâce à la suite spectrale de Cech.
 
 Il suffit donc d'établir l'isomorphisme suivant pour tout $\beta$

\begin{equation}\label{eqdr2}
\hdr{*} (\pi^{-1}(\XC_\eta (\frac{J}{\lambda})))\fln{\sim}{}\hdr{*} ( \pi^{-1}(\uni{\XC_\eta (\frac{J}{\lambda},\frac{\beta}{\left\lbrace i \right\rbrace}}{i\in J^c}{})))
\end{equation}
en montrant que cela découle de la condition imposée dans l'énoncé.

On a un recouvrement admissible à deux termes $$\pi^{-1}(\XC_\eta(\frac{J}{\lambda}))=\pi^{-1}(\XC_\eta (\frac{J}{\lambda}, \frac{J^c}{\beta})) \cup \pi^{-1}(\uni{\XC_\eta(\frac{J}{\lambda},\frac{\beta}{\left\lbrace i \right\rbrace})}{i\in J^c}{}).$$ En utilisant la suite exacte de Mayer-Vietoris qui en découle, il suffit d'établir\footnote{Soit un recouvrement admissible $X=U \cup V$, supposons la flèche $\hhh^i(U)\iso \hhh^i(U\cup V)$ bijective pour tout $i$. Par Mayer-Vietoris, on a une suite exacte $0\to \hhh^i(X)\to \hhh^i(U)\oplus \hhh^i(V)\to \hhh^i(U\cup V)\to 0$. Dit autrement, une classe sur l'union $X$ est équivalente à  une classe sur chaque ouvert $U$, $V$ qui coïncident sur l'intersection. Pour toute classe dans $V$, sa restriction à $U\cup V$ se relève de manière unique à $U$. En particulier, la classe de départ sur $V$ se relève de manière unique en une classe sur $X$. Cela établit l'isomorphisme voulue $\hhh^i(X)\iso \hhh^i( V)$.}
\[\hdr{*} (\pi^{-1}(\XC_\eta (\frac{J}{\lambda}, \frac{J^c}{\beta})))\fln{\sim}{}\hdr{*} (\pi^{-1}(\XC_\eta (\frac{J}{\lambda}, \frac{J^c}{\beta}) \cap \uni{\XC_\eta(\frac{J}{\lambda},\frac{\beta}{\left\lbrace i \right\rbrace})}{i\in J^c}{}))\]
On réécrit $$\pi^{-1}(\XC_\eta (\frac{J}{\lambda}, \frac{J^c}{\beta}) \cap \uni{\XC_\eta (\frac{J}{\lambda},\frac{\beta}{\left\lbrace i \right\rbrace})}{i\in J^c}{})=\pi^{-1}(\uni{\XC_\eta (\frac{J}{\lambda},\frac{\beta}{\left\lbrace i \right\rbrace}, \frac{J^c}{\beta}))}{i\in J^c}{}).$$
Pour établir l'isomorphisme ci-dessus, comparons les deux espaces considérés avec $\pi^{-1}(\XC_\eta (\frac{J}{\lambda}, (\frac{J^c}{\beta})^{\pm 1}))$ ie. prouvons la bijectivité des deux flèches 

\[\hdr{*} (\pi^{-1}(\XC_\eta (\frac{J}{\lambda}, \frac{J^c}{\beta})))\fln{\sim}{}\hdr{*} (\pi^{-1}(\XC_\eta (\frac{J}{\lambda}, (\frac{J^c}{\beta})^{\pm 1}))),\]

\[\hdr{*} (\pi^{-1}(\uni{\XC_\eta (\frac{J}{\lambda},\frac{\beta}{\left\lbrace i \right\rbrace}, \frac{J^c}{\beta}))}{i\in J^c}{}))\fln{\sim}{}\hdr{*} (\pi^{-1}(\XC_\eta (\frac{J}{\lambda}, (\frac{J^c}{\beta})^{\pm 1}))).\]Pour la première, cela  revient à comparer les deux couronnes de l'énoncé du lemme \ref{lemred2} pour la partition  $I=J\cup \emptyset\cup J^c$.
D'après la suite spectrale de Cech pour le recouvrement $(\XC_\eta (\frac{J}{\lambda},\frac{\beta}{\left\lbrace i \right\rbrace}, \frac{J^c}{\beta}))_{i\in J^c}$, on se ramène pour la deuxième à 
\[\hdr{*} (\pi^{-1}(\XC_\eta (\frac{J}{\lambda}, \frac{I\backslash (J\cup M)}{\beta},(\frac{M}{\beta})^{\pm 1} )))\fln{\sim}{}\hdr{*} (\pi^{-1}(\XC_\eta (\frac{J}{\lambda}, (\frac{J^c}{\beta})^{\pm 1})))\] pour tout $M\subset J^c$.
Là encore, cela revient à établir la condition suffisante du lemme pour la partition $I=J\cup M\cup (I\backslash (J\cup M))$

\end{proof}

\subsection{Fin de la preuve du théorème \ref{theoprinc}\label{ssectionfin}}

  Fixons une partition $\left\llbracket 1,r\right\rrbracket =J_1 \cup J_2 \cup  J_3$ et reprenons les notations introduites dans le lemme \ref{lemred2}. Dans toute la suite, nous appellerons par abus $\TC$ le revêtement sur $\tilde{C}$ et $\TC'$ celui sur $\tilde{C}'$. Nous devons comparer les cohomologies de de Rham de $\TC$ et de $\TC'$.  Pour cela, on consid\`ere les inclusions $\tilde{C}' \to \tilde{C} \to ]Y_I[_\XC =\left\lbrace s \in \XC_\eta, \forall i\in I, |x_i^*(s)|<1\right\rbrace$. 
Nous allons commencer par une description plus simple de $  ]Y_I[_\XC$, fournie par: 
  
\begin{lem}
Il existe une $\OC_L$-algèbre $\varpi$-adiquement complète et formellement lisse $\hat{B}$ et un isomorphisme 
 \[]Y_I[_\XC \cong \spg(\hat{B} \otimes_{\OC_L} L) \times \{Z=(Z_1,\cdots, Z_r)\in \mathring{\B}^{r}_L :Z^{\alpha}=\varpi\} \] 
 envoyant $Z_i$ sur $x_i^*$.
 \end{lem}
\begin{proof}

Le morphisme \'etale $\OC_L\langle X_1,...,X_d\rangle/ (X_1^{\alpha_1}... X_r^{\alpha_r}- \varpi ) \rightarrow A$ induit, en 
   compl\'etant $(X_1,...,X_r)$-adiquement,  un morphisme \'etale 
   $$ R:=  \OC_L\left\llbracket X_1,...,X_r\right\rrbracket\langle X_{r+1},...,X_d\rangle/ (X_1^{\alpha_1}... X_r^{\alpha_r}- \varpi ) \rightarrow \hat{A}.$$ D'après  \cite[(0.2.7)PROPOSITION]{Bert2}, $]Y_I[_\XC=\spf(\hat{A})^{rig}$. 
   Considérons le diagramme commutatif suivant, 
   
\begin{center}

   \begin{tikzcd}   
  \hat{A}\arrow{r} &  \hat{A}/\varpi \arrow{r}  & \bar{B}   \\
 R \arrow{r}\arrow{u} &    R/\varpi    \arrow{r}{\theta} \arrow{u}  & \F[X_{r+1},...,X_d] \arrow{u}
   \end{tikzcd}
\end{center}
dans lequel $\theta$ est la projection modulo $J=(X_1,...,X_r )$ et  
   $$\bar{B}:= \hat{A}/ \varpi \otimes_{R/\varpi}  \F[X_{r+1},...,X_d]=\hat{A}/ (\varpi\hat{A}+J\hat{A}).$$
 
Observons que l'inclusion canonique $\iota: \F[X_{r+1},..., X_d]\hookrightarrow R/\varpi$ est une section de $\theta$. Comme $R\rightarrow \hat{A}$ est \'etale, la section $\iota$ se rel\`eve en une section   $s: \bar{B}\rightarrow \hat{A}/\varpi$. En effet, $\bar{B}$ est lisse sur $\F$ car étale sur $\F[X_{r+1},...,X_d]$ par changement de  base de $R/\varpi\to \hat{A}/\varpi$. Ainsi, le morphisme naturel $\bar{B}\iso \hat{A}/(\varpi\hat{A}+J\hat{A})$ se relève en un morphisme $\bar{B}\to \hat{A}/\varpi$ car $\hat{A}/\varpi$ est complet pour la topologie $J$-adique. Pour vérifier que l'on obtient bien la section recherchée, il suffit de montrer que ce morphisme est compatible à $\iota$ en le réduisant modulo $J$ ce qui est vrai par construction. 

Par le théorème d'Elkik \cite[THÉORÈME fin section II p568]{elk1}, on peut relever $\bar{B}$ en une 
 $\OC_K$-alg\`ebre lisse $B$. Ainsi, en reprenant les arguments de la construction de $s$, on voit que $B \to \hat{A}/ \varpi$ se rel\`eve en un morphisme $B \rightarrow \hat{A} $ puis en $\hat{B} \rightarrow \hat{A} $, où $\hat{B}$ est la compl\'etion $\varpi$-adique de $B$.   Le diagramme commutatif
\begin{center}
\begin{tikzcd}
\hat{A}\arrow{d} & \hat{B}\arrow{l}\arrow{d} \\
\hat{A}/\varpi & \bar{B}\arrow[l] 
\end{tikzcd}
\end{center}
fournit un morphisme  $\beta: \hat{B}\left\llbracket Z_1,...,Z_r\right\rrbracket/ (Z_1^{\alpha_1}... Z_r^{\alpha_r}-\varpi) \rightarrow \hat{A}$ envoyant $Z_i$ sur $x_i^*$. Par Nakayama topologique, $\beta$ est un isomorphisme car il l'est modulo $(Z_1,...,Z_r,\varpi)$ par construction. On conclut en passant \`a la fibre  g\'enerique.

\end{proof}

Notons $S=\spg(\hat{B} \otimes_{\OC_L} L)$ et $X=\{Z=(Z_1,\cdots, Z_r)\in \mathring{\B}^{r}_L :Z^{\alpha}=\varpi\} $. Le lemme ci-dessus fournit 
  un diagramme commutatif \[
\xymatrix{ 
\tilde{C}' \ar[r] \ar[d]^{\rotatebox{90}{$\sim$}}  &  \tilde{C} \ar[r] \ar[d]^{\rotatebox{90}{$\sim$}}  & ]Y_I[_\XC \ar[d]^{\rotatebox{90}{$\sim$}}  \\
S\times  C' \ar[r] & S\times C \ar[r]& S\times X}
\] 
 o\`u $C$ et $C'$ sont les espaces par
\[ C=\left\lbrace s\in X |\, |Z_{j_1} (s)|\leq \lambda, |Z_{j_2} (s)|=\beta, |Z_{j_3} (s)|\leq\beta,\, \forall j_k\in J_k\right\rbrace ,\]\[ C'=\left\lbrace s\in X|\,  |Z_{j_1} (s)|\leq \lambda, |Z_{j_2} (s)|=\beta, |Z_{j_3} (s)|=\beta,\, \forall j_k\in J_k\right\rbrace.\]

On se place maintenant dans le cas où $\XC$ est semi-stable ie. $\forall i\in I, \alpha_i=1$. Alors, en exprimant $Z_1$ en fonction des variables $Z_2, \cdots,  Z_{r}$, nous pourrons voir les espaces $C$, $C'$ comme des tores monomiaux géométriquement connexes. Comme on peut raisonner sur chaque composante de $S$, on suppose de même $S$ connexe. Enfin, quitte \`a \'etendre les scalaires, on suppose le corps de base $L$ complet alg\'ebriquement clos. 

\begin{lem}\label{lemprod}
Les torseurs $[\TC]$ et $[\TC']$ admettent des décompositions en sommes $[\TC_1\times C]+[S\times\TC_2]$ et $[\TC_1'\times C']+[S\times\TC_2']$ où $\TC_1=\TC_1'$ (resp. $\TC_2$, $\TC_2'$ ) est un $\mu_n$-torseur sur $S$ (resp. $C$, $C'$). 
\end{lem}
\begin{proof}

On raisonne sur $C$, l'argument sera le même pour $C'$. 
Introduisons le diagramme 
\[
\begin{tikzcd}
S\times C \ar[r, "pr_2"] \ar[d, "pr_1"'] \ar[rd, "\varphi"] & C \ar[d, "\varphi"] \\ S \ar[r, "\varphi"] & \spg L.
\end{tikzcd}
\]

On se donne $c_0 \in C$ un point g\'eom\'etrique et on note $\iota : c_0 \times S \to C \times S$. D'apr\`es la suite spectrale de Leray, on a une suite exacte : 
\[ 0 \to \rrr^1\varphi_* \mu_{n,S}^{\pi_0(C)} \xrightarrow{\rrr^1pr_1^*} \rrr^1 \varphi_* \mu_{n,S \times C} \to \varphi_* \rrr^1 pr_{1,*}\mu_{n,S \times C} \to \rrr^2 \varphi_* \mu_{n,S}^{\pi_0(C)} \xrightarrow{\rrr^2pr_1^*} \rrr^2 \varphi_* \mu_{n,S\times C}. \]
Les morphismes induits $\rrr^1 \iota^*$ et $\rrr^2 \iota^*$ fournissent des inverses \`a droite de $\rrr^1pr_{1}^*$ et $\rrr^2 pr_{1}^*$, d'o\`u une suite exacte de faisceaux scind\'ee : 
\[ 0 \to (\rrr^1\varphi_* \mu_{n,S})^{\pi_0(C)} \xrightarrow{\rrr^1pr_1^*} \rrr^1 \varphi_* \mu_{n,S \times C} \to \varphi_* \rrr^1 pr_{1,*}\mu_{n, S \times C} \to 0 \]
et une identification $\varphi_* \rrr^1 pr_{1,*}\mu_{n,S \times C} \cong \ker(\rrr^1 \iota^*)$. 

Les flèches naturelles compatibles $\het{1}(C, \mu_n)\to \het{1}(U\times C, \mu_n)$ induisent un morphisme entre le préfaisceau constant $\het{1}(C, \mu_n)|_S$ et le préfaisceau $U\mapsto \het{1}(U\times C, \mu_n)$. En passant au faisceau associé, on obtient un morphisme naturel $\delta :\het{1}(C, \mu_n)|_S\to \rrr^1 pr_{1,*}\mu_{n,S \times C}$.
Le faisceau $pr_{1,*} \mu_{n,S \times C}$ est surconvergent en tant qu'image directe d'un faisceau constant (à quelques txists à la Tate près), on a d'apr\`es \cite[Th. 3.7.3]{djvdp}  $(\rrr^1 pr_{1,*}\mu_{n,S \times C})_s\cong\het{1}(C, \mu_n)$ pour tout point $s\in S$ et $\delta$ est un isomorphisme. 

Comme $S$ et $C$ sont connexes, on a, en prenant les sections globales, une suite exacte scind\'ee 
\[ 0 \to \het{1}(S, \mu_n) \to \het{1}(S \times C, \mu_n) \to \het{1}(C, \mu_n) \to 0 \]
avec $\het{1}(C, \mu_n) \iso \ker(\iota^*)$. Mais, on a un morphisme injectif $\het{1}(C, \mu_n) \to \het{1}(S \times C, \mu_n)$ induit par $pr_{2}^*$. Il reste \`a prouver que $\ker(\iota^*)= {\rm im}(pr_{2}^*)$. Le morphisme $pr_2^*$ envoie un rev\^etement $\TC$ sur $S \times \TC$ et $\iota^*$ envoie $\TC'$ sur $S \times C$ vers sa restriction \`a $S \times c_0 \cong S$. Il est ais\'e de voir que $\iota^* \circ pr_2^*=0$ i.e. ${\rm im}(pr_{2}^*)\subset \ker(\iota^*)$. Mais ces deux groupes ont tous deux pour ordre $|\het{1}(C, \mu_n)|<\infty$ (cf \ref{proph1etmonom} et \ref{remkummrev}), ils sont donc confondus. On a donc la d\'ecomposition 
\[ \het{1}(S \times C, \mu_n) = {\rm im}(pr_{1}^*) \oplus {\rm im}(pr_{2}^*). \]   

L'énoncé  est une traduction en termes de torseurs de cette égalité.

\end{proof}



L'espace $\TC_1\times_L\TC_2$ est un revêtement de $\TC$ de groupe de Galois $H=\{(g,g^{-1}) : g\in \mu_n\}$ et un revêtement de $S\times C$ de groupe de Galois $\mu_n^2$. On obtient l'égalité par Künneth : 

\begin{align*}
\hdr{q}(\TC) = \hdr{q}(\TC_1\times_L\TC_2)^H &= ( \drt{\hdr{q} (\TC_1\times \TC_2)[ \chi_1, \chi_2] }{\chi_1, \chi_2\in \mu_n^\lor }{}))^H\\
&=\drt{\drt{\hdr{q_1} (\TC_1)[ \chi]\otimes \hdr{q_2} (\TC_2)[ \chi] }{\chi\in \mu_n^\lor }{}}{ q_1 +q_2=q}{} 
\end{align*}
 De même pour $\TC_1\times_L\TC_2'$ par rapport à $\TC'$ et $S\times C'$, on a  $$\hdr{q} (\TC') =  \drt{\drt{\hdr{q_1} (\TC_1)[ \chi]\otimes \hdr{q_2} (\TC_2')[ \chi] }{\chi\in \mu_n^\lor }{}}{ q_1 +q_2=q}{}.$$


On en déduit le diagramme commutatif 
\[
\xymatrix{ 
\hdr{q}(\TC) \ar@{=}[r] \ar[d]^{}  &  \drt{\drt{\hdr{q_1} (\TC_1)[ \chi]\otimes \hdr{q_2} (\TC_2)[ \chi]}{\chi\in \mu_n^\lor }{}}{ q_1 +q_2=q}{}  \ar[d]^{\oplus_{q_1,q_2, \chi} ({\rm Id}\otimes \iota^*)}  \\
\hdr{q}(\TC')  \ar@{=}[r] & \drt{\drt{\hdr{q_1} (\TC_1)[ \chi]\otimes \hdr{q_2} (\TC_2')[ \chi] }{\chi\in \mu_n^\lor }{}}{ q_1 +q_2=q}{}
} .
\] 
Mais d'après \ref{coroinclqsiso}, la flèche $\iota^*$ est un isomorphisme et les deux flèches verticales sont des bijections. On peut alors appliquer \ref{lemred2} pour prouver le théorème \ref{theoprinc}.

Supposons maintenant $\XC$ semi-stable généralisé et $\pi={\rm Id}$. On considère le tore monomial \[Y=\{\tilde{Z}=(T, Z_1,\cdots, Z_d) : T\prod_{i=0}^d Z_i^{\alpha_i}=\varpi \et |T|\le \lambda^{1/\alpha_0}, |Z_{j_1} (s)|\leq \lambda, |Z_{j_2} (s)|=\beta, |Z_{j_3} (s)|\leq\beta\, \forall j_k\in J_k\}\] et on peut écrire $C=Y(T^{1/\alpha_0})$. En raisonnant de même sur $C'$, on observe l'isomorphisme $\hdr{*}(C)\cong \hdr{*}(C')$ d'après \ref{coroinclqsiso} et on en déduit $\hdr{*}(S\times C)\cong \hdr{*}(S\times C')$ par Künneth.

\section{Cohomologie de De Rham du premier rev\^etement de la tour de Drinfeld\label{ssectionhdrdr}}

  Le but de ce chapitre est de calculer la "partie cuspidale" de la cohomologie de de Rham de $\Sigma^1$. Cela utilise tous les résultats obtenus jusqu'à présent. Dans tout ce chapitre nous noterons 
  $$N=q^{d+1}-1, \,\breve{K}_N=\breve{K}(\varpi^{\frac{1}{N}}).$$

\begin{theo}\label{theodrsigma1}
Soit $\theta: \F_{q^{d+1}}^*\to \breve{K}_N^*$ un caract\`ere primitif. Il existe un isomorphisme naturel
\[ \hdrc{d}(\Sigma^1_{\breve{K}_N})[\theta]\simeq \bigoplus_{s \in \BC \TC_0} \hrigc{d}(\dl_{\F}^d/\breve{K}_N)[\theta]
\]
et $\hdrc{r}(\Sigma^1_{\breve{K}_N})[\theta]=0$ pour $r\ne d$.
\end{theo}

   Pour démontrer le théorème, on étudie la 
 suite spectrale de Cech associée au recouvrement par les tubes (dans $\Sigma^1$) au-dessus des composantes irréductibles de la fibre spéciale de $\H_{\OC_{\breve{K}}}^d $ i.e. au recouvrement par la famille d'ouverts $(\Sigma^1_{ \ost(s)})_{s\in \BC\TC_0}$ :
\[E_1^{-r,s}=\bigoplus_{\sigma\in\BC\TC_r}\hdrc{s}(\Sigma^1_{\ost(\sigma)}) \Rightarrow \hdrc{s-r} (\Sigma^1).\]
Par dualit\'e de Poincaré \ref{theodualitepoinc}, on se ram\`ene \`a \'etudier $\hdr{r}(\Sigma^1_{\ost(\sigma)})$. 
Nous montrerons (cf. paragraphe \ref{sssectiondrsigma1}) que, pour $s\in \BC\TC_0$, on a un isomorphisme naturel 
$$\hdrc{r}(\Sigma^1_{\breve{K}_N, \ost(s)})[\theta] = \begin{cases} \hrigc{d}(\dl_{\F}^d/\breve{K}_N )[\theta] & \text{si } r=d \\ 0 & \text{sinon} \end{cases}$$
et (cf. paragraphe \ref{sssectiondrsimpsigma1}) que $\hdrc{r}(\Sigma^1_{\ost(\sigma)})[\theta] =0$ pour $\mathrm{dim}(\sigma) \ge 1$. Ces deux résultats, dont la preuve utilise de manière cruciale le théorème $\ref{theoprinc}$, montrent 
la d\'eg\'en\'erescence de la suite spectrale et permettent de conclure la preuve du théorème \ref{theodrsigma1}.



\subsection{Le tube au-dessus d'une composante irréductible \label{sssectiondrsigma1}}

   Le but de ce paragraphe est de calculer la cohomologie de De Rham du tube au-dessus d'une composante irr\'eductible i.e. $\hdr{r}(\Sigma^1_{ \ost(s)})$, plus précisément de démontrer le
   résultat suivant: 

\begin{prop}\label{lemdrsommetsigma1}
Si $s\in\BC\TC_0$ est un sommet, il existe un isomorphisme naturel $$\hdrc{r}(\Sigma^1_{\breve{K}_N, \ost(s)})[\theta] = \begin{cases} \hrigc{d}(\dl_{\F}^d/\breve{K}_N )[\theta] & \text{si } r=d \\ 0 & \text{sinon} \end{cases}.$$
\end{prop} 

D'après le théorème $\ref{theoprinc}$ et la discussion dans le paragraphe précédent on a un isomorphisme naturel (induit par la restriction)  
   $$\hdr{r}(\Sigma^1_{ \ost(s)})\simeq \hdr{r}(\Sigma^1_{ s}).$$
   
   Pour étudier $ \hdr{r}(\Sigma^1_{ s})$, nous devons rendre explicite le lien entre $\Sigma^1_s$ et la vari\'et\'e de Deligne-Lusztig $\dl_{\bar{\F}}^d$. Ce lien est établi dans \cite[ 2.3.8]{wa}, mais nous allons donner l'argument pour le confort du lecteur. 

\begin{lem} \label{lemlieulisse}
La restriction $\Sigma^1_{\breve{K}_N, s}$ du premier rev\^etement au-dessus d'un sommet admet un mod\`ele entier lisse $\widehat{\Sigma}_s^{1}$ dont la fibre sp\'eciale $\overline{\Sigma}_s^1$ est isomorphisme la vari\'et\'e de Deligne-Lusztig $\dl_{\bar{\F}}^d$. 

De plus, l'isomorphisme ci-dessus est $\gln_{d+1}(\OC_K)\times \F_{q^{d+1}}^*$-équivariant.
\end{lem} 

Grâce au lemme précédent et à \ref{theopurete} on obtient des isomorphismes $\gln_{d+1}(\OC_K)\times \OC_D^*/(1+ \Pi_D \OC_D)$-équivariants \[\hdr{r}( \Sigma^1_{\breve{K}_N, s}) \cong \hrig{r}(\overline{\Sigma}^1_{s}/\breve{K}_N) \cong \hrig{r}(\dl_{\F}^d/\breve{K}_N),\]
  ce qui termine la preuve de la proposition \ref{lemdrsommetsigma1}.
  
\begin{proof}
On peut supposer que $s$ est le sommet standard. 
L'énoncé du théorème \ref{theoeqsigma1} introduit une fonction inversible $u_1$ sur $\H_{\breve{K}_N,s}^d$ voire même sur $\H_{\OC_{\breve{K}_N},s}^d$ telle que $\Sigma^1_{\breve{K}_N,s}= \H_{\breve{K}_N,s}^d(u_1^{\frac{1}{N}})$. La normalisation de $\H^d_{\OC_{\breve{K}},s}$ dans $\Sigma^1_{\breve{K}_N,s}$ fournit un mod\`ele entier $\widehat{\Sigma}_{s}^{1}=\H_{\OC_{\breve{K}_N},s}^d(u_1^{\frac{1}{N}})$ de fibre sp\'eciale 
\[ \overline{\Sigma}^1_{s}= \H_{\bar{\F},s}^d(u_1^{\frac{1}{N}}). \] 
D'après \ref{propdleq}, le revêtement de type Kummer $ \H_{\bar{\F},s}^d(u_1^{\frac{1}{N}})$ admet un isomorphisme $\F_{q^{d+1}}^*$-équivariant vers $\dl_{\bar{\F}}^d$. Le reste de la preuve consiste à montrer qu'il est de plus $G^{\circ}:=\gln_{d+1}(\OC_K)$-équivariant.


Rappelons que toutes les actions possibles de $G^{\circ}$ sur $ \bar{\Sigma}^1_s$ ou sur ${\Sigma}^1_s$ commutant avec le revêtement se déduisent l'une de l'autre en tordant par un caractère \[\chi \in \homm (G^{\circ},\mu_N(\bar{\F}))\cong\F^*,\] d'après \cite[Remarque 4.12]{J3}. Plus précisément,  pour $T$ une racine $N$-ième de $u_1$ dans $\Of(\bar{\Sigma}^1_s )=\Of(\dl_{\bar{\F}}^d )$ ou  dans $\Of({\Sigma}^1_s )$, le caractère $\chi$ associé à deux actions de $G^{\circ}$ notées $g\mapsto [g]_1$ et $g\mapsto [g]_2$ est donné par $\chi(g)=[g]_1\cdot T / [g]_2 \cdot T $. On en déduit aussi que l'on a une bijection entre les actions de $G^{\circ}$ sur  $ \bar{\Sigma}^1_s$  et celles sur ${\Sigma}^1_s$ qui préserve les caractères $\chi$.

Nous étudions maintenant le cas où  $g\mapsto [g]_1$ (resp.  $g\mapsto [g]_2$) est l'action provenant de $\Sigma^1$ (resp. de $\dl_{\F}^d$). Il s'agit de prouver que le caractère obtenu est trivial. On voit aisément que ce dernier se factorise via la flèche $G^{\circ} \fln{\det }{}  \OC_{K}^* \to \F^*$ car $1+\varpi \OC_K$ est $N$-divisible.  Comme $\F^*$ est cyclique, il suffit de raisonner sur une matrice $g\in G^{\circ}$   bien choisie telle que $\det g$ engendre $\F^*$ et sur la fibre d'un point $y\in \H_{\OC_{\breve{K}}}^d( \OC_{\breve{K}})$ fixé par $g$.  

Expliquons comment choisir cette matrice.  Observons que la norme $N: \F_{q^{d+1}}^*\to \F^*$  est surjective \footnote{Il s'agit de la flèche $x\mapsto x^{\tilde{N}}$.}. On se donne $\lambda \in \F_{q^{d+1}}$ un élément dont la norme engendre $\F^*$ si bien que $\F_{q^{d+1}}= \F[\lambda] $ et le polynôme minimal sur $\F$ de $\lambda$ est de degré $d+1$.  De plus, n'importe quel relevé unitaire $P$ sur $\OC_{K}$ est encore irréductible de degré $d+1$.  Donnons-nous  une matrice $g\in \mat_{d+1}(\OC_K)$   de polynôme caractéristique $P$ (par exemple la matrice compagnon associée). Ainsi, $g$ est en particulier dans $G^{\circ}$, régulière elliptique  et son déterminant engendre $\F^*$.



L'élément $g$ admet $d+1$ points fixes distincts dans $\P^{d}(\OC_{C})$ qui correspondent aux droites propres de $g$ dans $C^{d+1}$.  Montrons qu'elles sont toutes dans $\H^d_{K,s}(C)$. Pour cela, pour tout vecteur propre $v\in  C^{d+1}$ que l'on suppose unimodulaire à normalisation près,      nous devons prouver que  les coordonnées de $v$ modulo $\varpi$ sont libres sur $\F$.   Supposons que ce ne soit pas le cas et prenons $a\in \F^{d+1}\backslash \{0\}$  tel que $\langle a, v \rangle=0$.  Ainsi
\[
\langle (g)^t a ,  v \rangle= \langle a , g v \rangle = \lambda \langle a, v \rangle =0
\]
et $v^{\perp}\cap \F^{d+1}$ est  stable sous l'action de $g^t$.   Comme le polynôme caractéristique $P$ de $g^t$ est irréductible sur \footnote{Les sous-espaces stables sont en bijection avec les facteurs de $P$.} $\F$, $v^{\perp} \cap \F^{d+1}= \F^{d+1}$  ie. $v=0$, d'où une contradiction.

Prenons $y\in \H^d_{\OC_{\breve{K}},s}(\OC_{C}) $  le point fixe de $g$ correspondant à la valeur propre $\lambda$ et   $\bar{y}\in \H^d_{\bar{\F},s }(\bar{\F})$  sa spécialisation. D'après \ref{propdleq}, un point $x\in\dl_{\bar{\F}}^d (\bar{\F})$ corresponds à un vecteur $x=(x_0,\ldots, x_d)$ dans $\bar{\F}^{d+1}$ tel que $\tilde{u}_1(x)=1$. De plus, $x$ est dans la fibre $ \pi^{-1}(\bar{y})$,  si $x$ engendre la droite $\bar{y}$. Dans ce cas,  $[g]_2\cdot x = \lambda x$. Mais d'après la  description de l'isomorphisme $\dl^d_{\bar{\F}}(\bar{\F}) \iso \H_{\bar{\F},s}^d(u_1^{\frac{1}{N}})$ donnée dans \ref{remdlkumm}, on a $T=\frac{1}{x_0}$ et donc $[g]_2\cdot T= \lambda^{-1}T$ sur $ \pi^{-1}(\bar{y})$.


  Decrivons maintenant l'action provenant du premier revêtement $\Sigma^1$. D'après\footnote{\'Enoncé que l'on peut déjà trouver dans la preuve de \cite[(2.3.7)]{wa}} \cite[Corollaire 3.10]{J3},  la restriction $\XG[\Pi_D]_s$ des points de torsion du module universel $\XG[\Pi_D]$ au dessus du sommet $s$ est affine de sections  
  \[
  \Of(\XG[\Pi_D]_s)= \Of(\H^d_{\OC_{\breve{K}},s})[\tilde{T}]/ (\tilde{T}^{N+1}-\varpi u_1 \tilde{T}).
  \] Rappelons  que $\Of({\Sigma}_s^1)=\Of(\XG[\Pi_D]_s)[1/\tilde{T}, 1/\varpi]$ et $\tilde{T}=\varpi_N T$ dans $\Of({\Sigma}_s^1)\otimes \breve{K}_N$. Ainsi, calculer $[g]_1 \cdot\tilde{T}$ dans $\Of(\XG[\Pi_D]_s)$ revient à déterminer $[g]_1 \cdot T$ dans $\Of({\Sigma}_s^1)$ ou dans $\Of(\bar{\Sigma}_s^1)$. Plaçons-nous donc sur $\Of(\XG[\Pi_D]_s)$ et plus précisément sur l'idéal d'augmentation $\IC=\bigoplus_{i = 1}^{N} \tilde{T}^i \Of(\H^d_{\OC_{\breve{K}},s})$ de $\Of(\XG[\Pi_D]_s)$. La description précédente de l'idéal d'augmentation est en fait la décomposition en partie isotypique pour l'action de $\OC_{(d+1)}^*(\subset\OC_{D}^*)$. 
  Le cotangent de $\XG[\Pi_D]_s$ admet une décomposition similaire \[\IC/\IC^2=\bigoplus_{i\in 1}^{N} (\tilde{T}^i \Of(\H^d_{\OC_{\breve{K}},s})/\tilde{T}^i \Of(\H^d_{\OC_{\breve{K}},s})\cap \IC^2).\] Mais, on a aussi \[\IC/\IC^2 =(\lie \XG[\Pi_D]_s)^\vee=\ker (\Pi : \lie \XG_s \to \lie \XG_s)^\vee=(\lie \XG_s)^\vee / \Pi_{D,*}(\lie \XG_s)^\vee,\] il suffit donc de montrer que $\tilde{T} \Of(\H^d_{\OC_{\breve{K}},s})/\tilde{T} \Of(\H^d_{\OC_{\breve{K}},s})\cap \IC^2$ est non trivial et de comprendre l'action de la matrice $g$ fixée sur la  partie isotypique de $(\lie \XG_s )|_{\bar{y}}$ correspondante. 
  

  Utilisons la théorie de Cartier pour décrire le plan tangent.  D'après \ref{Drrep}, un point fermé   $z\in \H^d_{\bar{\F}}(\bar{\F})$  correspond à un couple\footnote{Le problème modulaire fait aussi intervenir une rigidification $\psi$ que nous ne mentionnerons pas.} 
  $(X,\rho)$ où $X$ est un $\OC_D$-module formel spécial sur $\bar{\F}$  et $\rho$ une quasi-isogénie  $  \Phi \rightarrow  X$ de hauteur $0$.  Aux modules $X$ et $\Phi$ correspondent des modules de Cartier $M_X $ et $M_{\Phi} $ sur\footnote{Cf.  \cite[ II.1.2 et II.1.3]{boca} la description de cette algèbre non-commutative.} $\OC_{\breve{K}}\langle F, V \rangle$ (ainsi qu'un opérateur $\Pi$ qui commute à $F$, $V$) qui sont libres de rang $4$ sur $\OC_{\breve{K}}$ et la quasi-isogénie $\rho$ induit un isomorphisme $M_{X}[\frac{1}{p}] \iso  M_{\Phi} [\frac{1}{p}]$.  Nous allons chercher à déterminer dans la suite les couples $(M_X,\rho)$ qui correspondent aux points de $\H^d_{\bar{\F},s}(\bar{\F})\subset \H^d_{\bar{\F}}(\bar{\F})$. Pour cela, nous aurons besoin de quelques résultats classiques sur les modules de Cartier spéciaux.

 L'action de $\OC_{(d+1)}$ sur $M_X$ et $M_{\Phi}$ induit des décompositions $M_X= \sum_{i\in \Z/(d+1)\Z} M_{X,i}$ (de même  pour $M_{\Phi}$) suivant les plongements de $\OC_{(d+1)}$ dans $\OC_{\breve{K}}$. Les opérateurs $V$, $\Pi$ sont de degré $1$ et $F$ est de degré $-1$ pour la graduation sur $\Z/(d+1)\Z$ précédente. On a le résultat classique suivant  :

\begin{lem}\label{lemmodcart}

Soit  $X/\bar{\F}$ un $\OC_D$-module formel spécial de dimension $(d+1)$  et de hauteur $(d+1)^2$ et $M_X$ le module de Cartier associé. On a les points suivants :

\begin{enumerate}
\item On a une identification $\lie X \cong M_X/VM_X$ et les parties isotypiques sont de la forme $(\lie X)_i=M_{X,i}/VM_{X,i-1}$
\item Pour tout $i\in \Z/(d+1)\Z$, on a $[M_{X,i},VM_{X,i-1}]=[M_{X,i},\Pi M_{X,i-1}]=1$ et $[M_{X,i},FM_{X,i+1}]=d$.
\item Soit $i\in \Z/(d+1)\Z$, les points suivants sont équivalents : 
       \begin{enumerate}
       \item $\Pi :(\lie X)_i\to (\lie X)_{i+1}$ est nulle,
       \item $\Pi M_{X,i}\subset VM_{X,i}$,
       \item $\Pi M_{X,i}=VM_{X,i}$,
       \item $M_{X,i+1}\neq \Pi M_{X,i} +VM_{X,i}$,
       \item $M_{X,i}=M_{X,i}^{V^{-1}\Pi}\otimes_{\OC_K}\OC_{\breve{K}}$ où $M_{X,i}^{V^{-1}\Pi}=\{m\in M_{X,i} : V m=\Pi m \}$.
       \end{enumerate}
Lorsque ces hypothèses sont vérifiées, on dit que $i$ est un indice critique.
\item Il existe au moins un indice critique.\end{enumerate}

\end{lem}

\begin{proof}
Pour $1.$, la deuxième assertion découle de la première qui est prouvée dans \cite[Th. 4.23]{zink}. Le reste a été prouvé dans \cite[§II.5]{boca} quand $d=1$. La traduction de ces arguments en dimension quelconque est transparente, mais nous en donnons quelques explications succincts. Pour $2.$ l'égalité $[M_{X,i},VM_{X,i-1}]=1$ provient du caractère spécial. Le reste se déduit du fait que $\Pi$, $V$, $F$ commutent entre eux, et des identités $[M_{X},\varpi M_{X}]=(d+1)^2$ (la hauteur de $X$) et $FV=\Pi^{d+1}=\varpi $. Pour $3.$, $(a)$ et $(b)$ sont équivalents d'après $1$. Le point $2.$ fournit le sens non-trivial de l'équivalence de $(c)$ avec $(a)$, $(b)$ (idem pour $(d)\Leftrightarrow (c)$). Pour $(c)\Rightarrow (e)$, on  applique la classification de Dieudonné-Manin à l'isocrystal $(M_{X,i}[1/p],V^{-1}\Pi)$ (l'autre sens est clair). Pour $4.$, on observe que la flèche $\Pi^{d+1}=\varpi : M_X/VM_X \to M_X/VM_X$ est nulle.
\end{proof}
  
  Pour simplifier, nous supposerons que tous les indices de $M_{\Phi}$ sont critiques.  Dans ce cas, un couple $(M_X,\rho)$ est dans  $\H^d_{\bar{\F},s}(\bar{\F})$ si et seulement si $\rho M_{X,d}=M_{\Phi,d}$ et $d$ est le seul indice critique de $M_X$. Pour un tel couple, on a une suite d'inclusion
\[\varpi M_{X,d}\subset \Pi^d M_{X,0}\subset M_{X,d}\]
qui définit une droite (d'après le point 4. et 2. de \ref{lemmodcart}) $D:=\Pi^d M_{X,0}/ \varpi M_{X,d}$ dans $M_{X,d}/ \varpi M_{X,d}$ et donc une flèche \[\psi :\H^d_{\bar{\F},s}(\bar{\F})\to \P(M_{X,d}/\varpi M_{X,d})\cong\P^d(\bar{\F}).\] On fixe aussi grâce au point 3. (e) de \ref{lemmodcart} une $\OC_K$-base de $M_{X,d}^{V^{-1}\Pi}$ qui donne lieu à une $\OC_{\breve{K}}$-base de $M_{X,d}$. Ceci confère à $\P(M_{X,d}/\varpi M_{X,d})$ un système de coordonnées et donc une notion convenable de droite et d'hyperplan $\F$-rationnels dans cet ensemble (qui ne dépends pas du choix de la base). Grâce à ces constructions, on peut décrire explicitement la restriction de l'isomorphisme dans \ref{Drrep} sur les points géométriques de $\H^d_{\bar{\F},s}$. 
  
  \begin{lem}
  
L'application $\psi$ définie précédemment induit une bijection $G^{\circ}$-équivariante:  
\[\H^d_{\bar{\F},s}(\bar{\F}) \iso \P^d(\bar{\F})\backslash \bigcup_{H\in \P^d(\F)}H.\]
  \end{lem}
  
\begin{proof}
Notons qu'un couple $(M_X,\rho)$  dans $\H^d_{\bar{\F},s}(\bar{\F})$ est déterminé par la donnée des inclusions $\rho M_{X,i}\subset   M_{\Phi} [\frac{1}{p}]$. Par hypothèse, il suffit d'étudier les cas $i\neq d$. Comme aucun de ces indices ne sont critiques, on a les identités (cf 3. (d) dans \ref{lemmodcart}) \[M_{X,1}=VM_{X,0}+\Pi M_{X,0}, \cdots , \ M_{X,d-1}=VM_{X,d-2} +\Pi M_{X,d-2}\]  et le sous-module $M_{X,0}$ détermine $M_{X,1}, \dots , M_{X,d-1}$. Revenons à $\psi$, la droite $\Pi^d M_{X,0}/ \varpi M_{X,d}\subset M_{X,d}/ \varpi M_{X,d}$ détermine $\rho\Pi^d M_{X,0}$ (car $\rho M_{X,d}$ et $\rho \varpi M_{X,d}$ sont fixés) et donc le module $\rho M_{X,0}$ par injectivité de la multiplication par $\Pi$. Dit autrement, l'application $\psi$ est injective. 

Il reste à prouver que l'image de cette flèche est l'ouvert décrit dans l'énoncé. Prenons un couple $(M_X,\rho)\in \H^d_{\bar{\F}}(\bar{\F})$ tel que $\rho M_{X, d}=M_{\Phi,d}$ et on fixe $m_0\in M_{X,0}\backslash V M_{X,d}$. En particulier, $m_0$ engendre $\lie (X)_0$. Le point crucial est d'observer l'équivalence suivante :   $d$ est le seul indice critique si et seulement si $(\Pi^i V^{d-i} m_0)_i$ est une $\bar{\F}$-base de  $M_{X,d}/\varpi M_{X,d}$. Supposons le deuxième hypothèse, le module suivant est contenue dans $V M_{X,d-1}$ et est d'indice $1$  dans $M_{X,d}$
\[\OC_{\breve{K}}\cdot V^d m_0+\cdots +\OC_{\breve{K}}\cdot V \Pi^{d-1}m_0  +\varpi M_{X,d}\]
Il est donc confondu avec $VM_{X,d-1}$ d'après le point 2. de \ref{lemmodcart}. Ainsi, $\Pi^d m_0$ engendre $\lie (X)_d$ et la flèche $\Pi^{ d} :\lie (X)_0\to\lie (X)_d$ est non nulle ce qui montre que $d$ est le seul indice critique. Réciproquement, supposons les indices $0,\dots, d-1$ non critiques et établissons d'abord par récurrence sur $0\le k\le d$ :
\[M_{X,k}= \OC_{\breve{K}}\cdot V^k m_0 +\OC_{\breve{K}}\cdot V^{k-1}\Pi m_0+\cdots +\OC_{\breve{K}}\cdot\Pi^k m_0 +V^{k+1} M_{X,d}, \]
\[V M_{X,k-1}=\OC_{\breve{K}}\cdot V^k m_0 + \OC_{\breve{K}}\cdot V^{k-1}\Pi m_0+\cdots +\OC_{\breve{K}}\cdot V\Pi^{k-1} m_0 +V^{k+1} M_{X,d}\]
Le cas $k=0$ est une conséquence directe de la définition de $m_0$. Pour l'hérédité, cela découle de la relation $M_{X,k+1}=\Pi M_{X,k}+V M_{X,k}$ pour $k\neq d$. Quand $k$ vaut $d$, on en déduit que la famille $(\Pi^i V^{d-i} m_0)_i$ engendre \[M_{X,d}/V^{d+1} M_{X,d}=M_{X,d}/\Pi^{d+1} M_{X,d}=M_{X,d}/\varpi M_{X,d}\] et est donc une base par argument de dimension.

Pour terminer la preuve,   écrivons $\Pi^d m_0=\sum_i a_i x_i$  dans une base $(x_i)_i$ de $M_{X,d}$ par des éléments de $M_{X,d}^{V^{-1}\Pi}$. On obtient alors dans $M_{X,d}/\varpi M_{X,d}$ :
\[\Pi^{d-i} V^i m_0=(\Pi^{-1}V)^i \Pi^d m_0=\sum_i a_i^{1/p^i} x_i\]
Le même argument que dans la preuve de \ref{propdleq} montre que $(\Pi^i V^{d-i} m_0)_i$ est une $\bar{\F}$-base de  $M_{X,d}/\varpi M_{X,d}$ si et seulement si $\bar{\F}\cdot V^d m_0 \in \P^d(\bar{\F})\backslash \bigcup_{H\in \P^d(\F)}H$ si et seulement si $\bar{\F} \cdot\Pi^d m_0 \in \P^d(\bar{\F})\backslash \bigcup_{H\in \P^d(\F)}H$.

\end{proof}

Nous appellerons l'isomorphisme précédent morphisme des périodes. Nous pouvons maintenant terminer la preuve du résultat. Reprenons la matrice $g$ construite précédemment ainsi que la droite propre $\bar{y}\in \H^d_{\bar{\F},s}(\bar{\F})$ associée à la valeur propre $\lambda$. Via l'application des périodes, cette droite détermine un couple $(M_X,\rho)$ pour lequel $g$ agit par multiplication par $\lambda$ sur $D:=\Pi^d M_{X,0}/ \varpi M_{X,d}$. 
En reprenant les arguments de la preuve du lemme précédent,  
les générateurs de $D$ engendrent $\lie (X)_d$ où $g$ agit donc par multiplication par $\lambda$ sur $\lie (X)_d$. Comme l'indice $d$ est critique, on a $\lie (X)_d =\ker (\Pi : \lie (X)_d\to\lie (X)_0)$  d'où $[g]_1\cdot T=\lambda^{-1} T$ sur $\pi^{-1} (\bar{y})$ par dualité. Le résultat s'en déduit.

\end{proof}



\subsection{Le tube au-dessus d'une intersection de composantes irréductibles \label{sssectiondrsimpsigma1}}
 
  Le but de ce paragraphe est de démontrer le résultat suivant:
  
  \begin{prop}\label{lemdrsimpsigma2}
  Si $\sigma$ est un simplexe de dimension non nulle, alors $\hdrc{j}(\Sigma^1_{\ost(\sigma)})[\theta] =0$ pour tout 
  $j$ et tout caractère primitif $\theta$. 
  \end{prop}

   En utilisant l'action de $G$, on peut supposer que $\sigma$ est un simplexe standard (en particulier, on a $\sigma=\{[M_i]\}_i$ avec $M_0=\OC_K^{d+1}$), de dimension $k\geq 1$ et de type 
   $(e_0,...,e_k)$. 
D'après le théorème \ref{theoprinc}, on dispose d'un isomorphisme naturel
 \[ \hdr{j} (\Sigma^1_{  \ost(\sigma)})[\theta] \cong \hdr{j} (\Sigma^1_{  \ost(\sigma)} \backslash (\bigcup_{s \notin \sigma} \Sigma^1_{ \ost(s)} ))[\theta] \cong \hdr{j} (\Sigma^1_{ \mathring{\sigma}})[\theta]. \]

On a,  d'apr\`es le théorème  \ref{theoeqsigma1} 
$$  \Sigma^1_{ \ost(s)}=\H_{\breve{K},\ost(s)}^d((\varpi u_1)^{\frac{1}{N}}),$$ avec $u_1= (-1)^{d}\prod_{a \in (\F_q)^{d+1} \backslash \{ 0 \}}\frac{a_0z_0 +\cdots +a_d z_d}{z_d}$ avec $z=[z_0,\cdots,z_d]$ un système de variable sur l'espace projectif $\P^d_K$  adapté à $\sigma$.  Nous souhaitons restreindre le torseur  $\kappa(\varpi u_1)$ \`a $\H_{\breve{K}, \mathring{\sigma}}^d$ car $\Sigma^1_{ \mathring{\sigma}}\subset\Sigma^1_{ \ost(s)}$. 
  On rappelle \cite[6.4]{ds} que cet espace admet une d\'ecomposition 
  \[ \H_{\breve{K}, \mathring{\sigma}}^d \cong A_k \times \prod_{i=0}^k C_{e_i-1}\cong A_k \times C_{\sigma}, \]
  où $A_k$ et $C_{e_i-1}, C_{\sigma}$ sont les espaces introduits dans le paragraphe \ref{paragraphhdkrectein}. Nous allons exhiber une décomposition similaire pour les torseurs. Introduisons avant quelques notations. Pour $a$ unimodulaire ie. $a\in M_0\backslash\varpi M_0$,   $i(a)$ sera  l'unique entier tel que $a \in M_{i(a)} \backslash M_{i(a)-1}$.
  
  \begin{lem}\label{lemdecompunitsimpcirc}
   Soient $S$ un ensemble fini de vecteur unimodulaire et $(\alpha_a)_{a\in S}$ des entiers de somme nulle (sans perte de généralité, on peut supposer que $S$ contient les variables $(z_i)_{0\le i\le d}$) et soit 
 $Q=\prod_{a\in S} l_a^{\alpha_a}$. Notons \[Q_{A_k}=\pro{z_{d_{i(a)}}^{\alpha_a}}{a\in S}{} \in \Of^*(A_k),\] 
 Il existe des sections $Q_{C_{e_i-1}}\in \Of^*(C_{e_i-1})$, chacune étant un produit homogène de degré $0$ de formes linéaire tel que  
  \[  Q= Q_{A_k} Q_{C_{e_0-1}} \dots Q_{C_{e_k-1}} =Q_{A_k} Q_{C_{\sigma}} \pmod {1+ \Of^{++}(\H_{\breve{K}, \mathring{\sigma}}^d)}  \]
 avec   $Q_{C_{\sigma}}=Q_{C_{e_0-1}} \dots Q_{C_{e_k-1}} \in\Of^*(C_{\sigma})$
  \end{lem}
  
  \begin{proof}
Il suffit de montrer le r\'esultat pour les fractions de la forme $\frac{l_a}{z_{d}}$. On rappelle que l'on a introduit des sous-modules $N_i\subset M_i$ tels que :
\[M_k=N_0\oplus \cdots\oplus N_k\]
\[M_i=N_0\oplus \cdots\oplus N_i\oplus \varpi N_{i+1}\oplus \cdots\oplus\varpi N_k.\] 

On peut alors d\'ecomposer $a = a_1 +a_2$ avec $a_1$ dans $N_{i(a)}$ unimodulaire, $a_2$ dans $M_{i(a)-1}$. On a 
  \[\frac{l_a}{z_{d}}= \frac{z_{d_{i(a)}}}{z_{d}}   \frac{l_{a_1}+ l_{a_2}}{z_{d_{i(a)}}}=\frac{z_{d_{i(a)}}}{z_{d}}   \frac{l_{a_1}}{z_{d_{i(a)}}} (1+ \frac{l_{a_2}}{l_{a_1}}). \]
 Le  terme $\frac{z_{d_{i(a)}}}{z_{d}}$ est dans $\Of^*(A_k)$, la fraction $\frac{l_{a_1}}{z_{d_{i(a)}}}$ dans $\Of^*(C_{e_{i}-1})$ (cf \ref{remniinvci}) et $\frac{l_{a_2}}{l_{a_1}}$ dans $\Of^{++}(\H_{\breve{K},\mathring{\sigma}}^d)$ par définition de $\H_{\breve{K},\mathring{\sigma}}^d$. De plus, la fraction $Q_{A_k}=\frac{z_{d_{i(a)}}}{z_{d}}$ a la forme voulue ce qui conclut la preuve de l'\'enonc\'e.   
  
  \end{proof}
  
  D'après ce qui précède (et en utilisant le fait que $N$ est premier à 
  $p$, donc toute fonction $f$ telle que $|f-1|<1$ est une puissance $N$-ième), on peut  décomposer $u_1=u_{A_k}u_{C_{\sigma}} \pmod{\Of^*(\H_{\breve{K},\mathring{\sigma}}^d)^N}$. Introduisons les $\mu_N$-revêtements 
  $$\TC_{A_{k}}=A_{k} (u_{A_k}^{1/N})\,\,\text{et}\,\, \TC_{C_{\sigma}}= C_{\sigma} ((\varpi u_{C_\sigma})^{1/N} ).$$
  Soit $H$ l'antidiagonale de $\mu_N^{2}$. On a 
  \[ \Sigma^1_{ \mathring{\sigma}} = (A_k \times C_{\sigma})((\varpi u_{A_k}u_{C_{\sigma}})^{1/N})=(\TC_{A_{k}} \times_{\breve{K}} \TC_{C_{\sigma}})/ H. \]
  
\begin{lem}
 L'espace $\TC_{A_{k}}$  a $q^m-1$ composantes géométriques avec $m={\rm PGCD}(d+1,e_0,\cdots,e_k)$. 

\end{lem}

\begin{proof}
Notons que $A_k$ est un  tore mon\^omial semi-ouvert. En effet, $$A_k=\{y=(y_0,\cdots, y_{k-1})\in \mathring{\B}_{\breve{K}}^k|\, 1>|y_{k-1}|>\cdots>|y_0|>|\varpi|\}$$ et on réalise le changement de variable $x_0=\frac{y_0}{y_1}$, $x_1=\frac{y_1}{y_2}, \cdots, x_{k-1}={y_{k-1}}$ pour obtenir\footnote{Sur le système de coordonnées $[z_0,\dots, z_d]$ de l'espace projectif  ambiant, les variables $(y_i)_i$ et $(x_i)_i$ correspondent   à $y_i=\frac{z_{d_i}}{z_{d}}$ et $x_i=\frac{z_{d_i}}{z_{d_{i+1}}}$.} $$A_k=\{x=(x_0,\cdots, x_{k-1})\in \mathring{\B}_{\breve{K}}^k|\, 1>|x_0\cdots x_{k-1}|>|\varpi|\}.$$ De plus, on a (cf \ref{lemdecompunitsimpcirc}) \[u_{A_k}=\pro{(\frac{z_{d_i}}{z_{d}})^{\beta_i}}{i=0}{k-1}=\pro{y_i^{\beta_i}}{i=0}{k-1}=\pro{x_i^{\beta_i+\cdots +\beta_{k-1}}}{i=0}{k-1}\]
où $\beta_i =|(M_i/\varpi M_0)\backslash (M_{i+1}/\varpi M_0)|=q^{d+1-d_{i+1}}(q^{e_i}-1)$.
 
 Le nombre de composantes connexes géométriques $\pi_0$ de $\TC_{A_{k}}$ est (cf \cite[Proposition 4.1.]{J3}) $$\pi_0={\rm PGCD} (N, \beta_0+\cdots +\beta_{k-1},\cdots,\beta_{k-1})={\rm PGCD} (N, \beta_0,\cdots,\beta_{k-1})=q^m-1.$$
\end{proof}

La partie isotypique en $\theta$ est donnée par : 
\[\hdr{j} (\Sigma^1_{  \mathring{\sigma}})[\theta]= \bigoplus_{j_1+j_2=j} \hdr{j_1}(\TC_{A_{k}}) [ \theta] \widehat{\otimes} \hdr{j_2}(\TC_{C_{\sigma}} )[ \theta]\]

On voit $\theta$ comme un élément de $\Z/N\Z$. Le nombre de composantes connexes géométriques de $\TC_{A_{k}}$ est de la forme $q^m-1$ et par primitivité de $\theta$, $N/(q^m-1)$ ne divise pas $\theta$ dans $\Z/N\Z$. On a alors $\hdr{j}(\TC_{A_{k}}) [ \theta]=0$ d'après  \ref{corosouv}. Ainsi, \[ \hdr{j} (\Sigma^1_{  \mathring{\sigma}})[ \theta ]=0. \]Ce qui conclut la preuve de la proposition \ref{lemdrsimpsigma2}.

\nocite{dr1}
\nocite{dr3}
  
\subsection{Réalisation de la correspondance de Langlands locale\label{sssectionsigdrjl}}

Dans cette partie, nous allons d\'ecrire la cohomologie des espaces $\MC_{Dr}^1$ et montrer qu'elle r\'ealise la correspondance de Jacquet-Langlands. On \'etendra les scalaires \`a $C$ pour tous les espaces consid\'er\'es en fibre g\'en\'erique.  


On pourra simplifier le produit $G \times D^* $ en $GD$. On a  
une "valuation" sur $GD$ :
\[ v_{GD} : (g,b) \in GD \mapsto v_K(\det(g) \nr(b) ) \in \Z. \]
On introduit alors pour  $i=0$ ou $i=d+1$, $[GD]_{i}=v_{GD}^{-1}(i\Z)$ et $[G]_{i}=G\cap[GD]_{i}$, $[D]_{i}=D^*\cap[GD]_{i}$.
 Ainsi, on a des inclusions  de $\OC_D^*$, $G$ dans $[GD]_0$
\[b\in  \OC_D^* \mapsto (Id,b)\in [GD]_0\]
\[g\in G\mapsto (g, \Pi_D^{- \det(g)})\in [GD]_0\]
mais les deux sous-groupes ne commutent pas entre eux. 

Passons aux repr\'esentations qui vont nous int\'eresser. Nous d\'efinissons d'abord des repr\'esentations sur $\gln_{d+1}(\OC_K) \varpi^{\Z}\times \OC_D^* \varpi^{\Z}$ que nous \'etendrons \`a $GD$ par induction. Fixons $\theta$ un caract\`ere primitif de $\F_{q^{d+1}}^*$ et des isomorphismes $\OC_D^*/1+ \Pi_D \OC_D \cong \F_{q^{d+1}}^* $. On pose : 
\begin{itemize}
\item $\theta$ sera vu comme une $[D]_{d+1}$-repr\'esentation via $\OC_D^* \varpi^{\Z} \to \OC_D^* \to \F_{q^{d+1}}^*$, 
\item $\bar{\pi}_{\theta}$ sera la repr\'esentation associ\'ee \`a $\theta$ sur $\gln_{d+1}(\F_q)$ via la correspondance de Deligne-Lusztig. 
On la voit comme une $\gln_{d+1}(\OC_K) \varpi^{\Z}$-repr\'esentation via $\gln_{d+1}(\OC_K) \varpi^{\Z} \to \gln_{d+1}(\F_q)$.
\end{itemize}  
Par induction, on obtient : 
\begin{itemize}
\item une repr\'esentation $\pi(\theta)$ de $G$, o\`u $\pi(\theta)= \cind_{\gln_{d+1}(\OC_K) \varpi^{\Z}}^G \bar{\pi}_{\theta}$. Il pourra \^etre utile de consid\'erer $\tilde{\pi}(\theta)= \cind_{\gln_{d+1}(\OC_K) \varpi^{\Z}}^{[G]_{d+1}} \bar{\pi}_{\theta}$ et d'\'ecrire $\pi(\theta)= \cind_{[G]_{d+1}}^G \tilde{\pi}(\theta)$. 
\item une $D^*$-repr\'esentation $\rho(\theta)= \cind_{[D]_{d+1}}^{D^*} \theta$, 
\end{itemize}
Nous avons d\'efini 
une action de $GD$ sur $\MC_{Dr}^1$ qui s'identifie non canoniquement \`a $\Sigma^1 \times \Z$. Si on confond  $\Sigma^1$ avec $\Sigma^1 \times \{0 \}\subset\MC_{Dr}^1/\varpi^\Z \cong \Sigma^1 \times \Z/(d+1)\Z$, on obtient une action sur $\Sigma^1$ de $[GD]_{d+1}$. 

Pour \'enoncer la correspondance de Langlands, nous aurons besoin de la cohomologie de $\MC_{Dr}^1$. On a la relation : 
\[ \hdrc{i}( \MC_{Dr}^1/ \varpi^{\Z}) = \cind^{GD}_{[GD]_{d+1}}  \hdrc{i}( \Sigma^1). \]


Nous allons montrer : 

\begin{theo}
Soit $\theta$ un caract\`ere primitif, on a un isomorphisme $G$-\'equivariant : 
\[ \homm_{D^*}(\rho(\theta), \hdrc{i}(( \MC_{Dr}^1/ \varpi^{\Z})/C)) \underset{G}{\cong} \begin{cases} \pi(\theta)^{d+1} &\text{ si } i=d, \\ 0 &\text{ sinon.} \end{cases} \]  
\end{theo}

\begin{proof}
Si $i \neq d$, nous avons d\'ej\`a prouv\'e l'annulation de la cohomologie. Posons dor\'enavant $i=d$. 
Dans un premier temps, observons 

\begin{align*}
\homm_{D^*}(\rho(\theta), \hdrc{d}(( \MC_{Dr}^1/ \varpi^{\Z})/C))  & = \homm_{D^*} (\cind_{[D]_{d+1}}^D \theta, \cind_{[GD]_{d+1}}^{GD} \hdrc{d}(\Sigma^1_{C}))&	 \\
& =  \homm_{[D]_{d+1}}(\theta, \cind_{[G]_{d+1}}^{G}\hdrc{d}(\Sigma^1_{C}))& \\
& = \cind_{[G]_{d+1}}^{G} \homm_{\F_{q^{d+1}}^*}(\theta, \hdrc{d}(\Sigma^1_{C}))& \\
& = \cind_{[G]_{d+1}}^{G}  \hdrc{d}(\Sigma^1_{C}))[\theta]& \\
& = \cind_{[G]_{d+1}}^G \pi(\theta) |_{[G]_{d+1}}&( \ref{theodrsigma1}) \\
& = \cind_{[G]_{d+1}}^G {\rm res}_{[G]_{d+1}} ( \cind_{[G]_{d+1}}^G \tilde{\pi}(\theta)) \\
& = \cind_{[G]_{d+1}}^G (\bigoplus_{x \in G/ [G]_{d+1}} c_x(\tilde{\pi}(\theta))) \\
& = \bigoplus_{x \in G/ [G]_{d+1}}\cind_{[G]_{d+1}}^G  c_x(\tilde{\pi}(\theta)) \\
& = \pi(\theta)^{| G / [G]_{d+1} |}= \pi(\theta)^{d+1}.
\end{align*}
On rappelle que $\tilde{\pi}(\theta) = \cind_{\gln_{d+1}(\OC_K) \varpi^{\Z}}^{[G]_{d+1}} \bar{\pi}_{\theta}$. On a not\'e $c_x(\tilde{\pi}(\theta))$ la repr\'esentation $ g  \mapsto \tilde{\pi}(\theta)(x^{-1}gx)$. Pour l'avant-derni\`ere \'egalit\'e, on a utilis\'e la formule de Mackey $\cind_{[G]_{d+1}}^G  c_x(\tilde{\pi}(\theta))=\pi(\theta)$.

\end{proof}

\newpage

\bibliographystyle{alpha}
\bibliography{drsig_v1}

\end{document}